\documentclass[final]{scrartcl}

\usepackage[utf8]{inputenc}
\usepackage{hyperref}
\usepackage[english]{babel}
\usepackage{enumerate}
\usepackage{aliascnt}
\usepackage{amssymb}
\usepackage{amsmath}
\usepackage{amsthm}
\usepackage{verbatim}
\usepackage[normalem]{ulem}
\usepackage{mathtools}
\usepackage{esint} 

\usepackage[noabbrev,capitalise]{cleveref}

\usepackage[right]{showlabels}

\usepackage{tikz}
\usepackage{tikzscale}
\usepackage{graphicx}
\usepackage[labelformat=simple]{subcaption}
\usepackage{mathscinet} 

\usepackage[nottoc]{tocbibind} 
\usepackage[title]{appendix}

\usepackage{color}

\numberwithin{equation}{section} 

\usepackage{dsfont} 
\usepackage{hyperref} 

\usepackage{autonum} 

\allowdisplaybreaks[4] 

\theoremstyle{plain}

\newtheorem{proposition}{Proposition}[section]

\newaliascnt{lemma}{proposition} 

\newtheorem{lemma}[lemma]{Lemma}

\aliascntresetthe{lemma} 
\Crefname{lemma}{Lemma}{Lemmas}

\newaliascnt{theorem}{proposition} 
\newtheorem{theorem}[theorem]{Theorem}

\aliascntresetthe{theorem} 

\newaliascnt{corollary}{proposition} 
\newtheorem{corollary}[corollary]{Corollary}

\aliascntresetthe{corollary}

\newaliascnt{hypothesis}{proposition}

\aliascntresetthe{hypothesis} 

\theoremstyle{definition}

\newaliascnt{definition}{proposition} 
\newtheorem{definition}[definition]{Definition}

\aliascntresetthe{definition} 
\Crefname{definition}{Definition}{Definitions}

\newaliascnt{problem}{proposition} 

\aliascntresetthe{problem} 

\newaliascnt{example}{proposition} 

\aliascntresetthe{example} 

\newaliascnt{assumption}{proposition} 

\aliascntresetthe{assumption} 

\theoremstyle{remark}

\newaliascnt{remark}{proposition} 
\newtheorem{remark}[remark]{Remark}

\aliascntresetthe{remark} 

\newtheorem{claim}{Claim}[proposition]

\makeatletter
\@addtoreset{claim}{proposition}
\makeatother

\def\equationautorefname~#1\null{%
	(#1)\null
}
 
\newcommand{\R}{\mathbb{R}}

\newcommand{\N}{\mathbb{N}}

\newcommand{\V}{\mathcal{V}}

\newcommand{\scurv}{\kappa}
\renewcommand{\S}{\mathbb{S}}

\newcommand{\A}{\mathcal{A}}

\renewcommand{\H}{\mathcal{H}}
\newcommand{\W}{\mathcal{W}}

\newcommand{\diam}{\mathrm{diam}}

\newcommand*{\dd}{\mathop{}\!\mathrm{d}}

\def\nicefrac#1#2{%
    \raise.5ex\hbox{$#1$}%
    \kern-.15em/\kern-.05em%
    \lower.25ex\hbox{$#2$}}

\title{Introducing spontaneous curvature to the Helfrich flow: Singularities and convergence}

\author{Manuel Schlierf\thanks{Institute of Applied Analysis, Ulm University, Helmholtzstra\ss e 18, 89081 Ulm, Germany. \texttt{manuel.schlierf@uni-ulm.de}}}

\date{April 19, 2024}

\begin{document}
\maketitle
\begin{abstract}
\noindent\textbf{Abstract:} While there are various results on the long-time behavior of the Willmore flow, the Helfrich flow with non-zero spontaneous curvature as its natural generalization is not yet well-understood. Past results for the gradient flow of a locally area- and volume-constrained Willmore flow indicate the existence of finite-time singularities which corresponds to the scaling-behavior of the underlying energy. However, for a non-vanishing spontaneous curvature, the scaling behavior is not quite as conclusive.

Indeed, in this article, we find that a \emph{negative} spontaneous curvature corresponds to finite-time singularities of the locally constrained Helfrich flow if the initial surface is energetically close to a round sphere. Conversely however, in the case of a \emph{positive} spontaneous curvature, we find a positive result in terms of the convergence behavior: The locally area-constrained Helfrich flow starting in a spherical immersion with suitably small Helfrich energy exists globally and converges to a Helfrich immersion after reparametrization. Moreover, this energetic smallness assumption is given by an explicit energy threshold depending on the spontaneous curvature and the local area constraint of the energy.

\end{abstract}

\bigskip
\noindent \textbf{Keywords and phrases:} Helfrich flow, Canham-Helfrich energy, Willmore energy, Willmore surfaces, Geometric flows, \L ojasiewicz-Simon inequality.

\noindent \textbf{MSC(2020)}: 53E40, 49Q10 (primary), 35B40, 35K41 (secondary).


\section{Introduction}

Since the works of Canham in \cite{canham1970} and Helfrich in \cite{helfrich1973}, the study of the shape of lipid bilayers and biological membranes, especially the shape of red blood cells, is closely related to the field of the calculus of variations. Indeed, they model the shape of such objects by studying the critical points of the so-called \emph{Canham-Helfrich functional}. Due to its analytical challenges, the variational study of this functional has remained an intriguing topic also for geometers and analysts with many contributions over the last decades. 

For an immersion $f\colon\Sigma\to\R^3$ of a closed, oriented surface $\Sigma$, $g_f=f^*\langle\cdot,\cdot\rangle$ denotes the pull-back metric on $\Sigma$, i.e. $g_{ij}=\langle\partial_i f,\partial_jf\rangle$ in local coordinates, with induced measure $\mu$. One also writes $\A(f)=\mu(\Sigma)$ for the \emph{area}. Further, $\nu\colon\Sigma\to\R^3$ is the unique smooth unit normal field induced by the orientation of $\Sigma$. In local coordinates with respect to the orientation of $\Sigma$, one has
\begin{equation}
  \nu=\frac{\partial_1 f\times\partial_2 f}{|\partial_1 f\times \partial_2 f|}.
\end{equation}
Throughout this article, given an immersion $f$, we always choose the orientation of $\Sigma$ such that the signed volume is non-negative, i.e.
\begin{equation}\label{eq:choice-of-orientation}
  \V(f)=-\frac{1}{3}\int_{\Sigma}\langle f,\nu\rangle\dd\mu \geq 0.
\end{equation}
Particularly, if $f$ is an embedding, then $\nu$ is the inward pointing normal field. The (scalar) second fundamental form of $f$ is the $(2,0)$-tensor determined by $A_{ij}=\langle\partial^2_{ij}f,\nu\rangle$ in local coordinates. Its trace is the mean curvature $H=g^{ij}A_{ij}$ where $(g^{ij}) = (g_{ij})^{-1}$. Particularly, if $\Sigma=\S^2$ and $f$ is a parametrization of a round sphere with radius $r>0$, $H\equiv\frac{2}{r}$. The trace-free second fundamental form is given by $A_{ij}^0=A_{ij}-\frac12 H g_{ij}$. For any $c_0\in\R$, define the \emph{Helfrich energy} with \emph{spontaneous curvature} $c_0$ by
\begin{equation}
  \H_{c_0}(f)=\frac{1}{4}\int_{\Sigma}(H-c_0)^2\dd\mu
\end{equation}
where in the special case $c_0=0$, $\W(f)=\H_{0}(f)$ is also called \emph{Willmore energy}. Also this special case of the Helfrich energy is already of high interest. Indeed, from a geometric point of view, the Willmore functional is conformally invariant which, from a PDE-perspective, introduces various challenges when investigating the associated Euler-Lagrange equation, the existence of minimizers or the associated $L^2(\dd\mu)$-gradient flow. In \cite{willmore1965}, Willmore shows some elementary properties of $\W$ and proposes his famous conjecture. Particularly, $\W(f)\geq 4\pi$ and one has equality only for round spheres. Further results on the minimization of the Willmore functional among closed surfaces, especially with prescribed genus, are obtained in \cite{simon1993,bauerkuwert2003}. An intriguing property of $\W$ and $\H_{c_0}$ from a PDE-perspective is the fact that one faces the critical case with respect to Sobolev embeddings. In the study of the minimization problem, some weaker concepts of surface or parametrization which come with a suitable compactness theory are required. A first concept in the direction of such weak notions is that of \emph{varifolds}. While it is rather well-suited for the analysis of the Willmore functional, e.g. cf. \cite{simon1993,kuwertmueller2022}, this weak notion of convergence comes with additional difficulties for the Helfrich functional with non-zero spontaneous curvature: even lower semi-continuity fails in general, cf. \cite{grossebrauckmann1993,eichmann2020,brazdalussardistefanelli2020}. One way of circumventing this issue is employing a concept of weak Sobolev immersions that was already used in the context of the Willmore energy in \cite{riviere2014,kuwertli2012}, cf. \cite{mondinoscharrer2020,ruppscharrer2023}.  

A fundamental geometric property of curvature functionals is the encoded topological information. In case of the Willmore energy, the Li-Yau inequality in \cite{liyau1982} shows that any immersion $f$ of a closed surface with $\W(f)<8\pi$ is already an embedding. While this yields also some immediate corollaries for the Helfrich functional if one uses ad-hoc bounds such as \eqref{eq:est-Will-energy}, in \cite{ruppscharrer2023}, a monotonicity-formula-based approach is used to deduce a Li-Yau inequality specifically for the Helfrich functional. 

For $c_0,\lambda\in\R$, the area-penalized Helfrich energy is given by
\begin{equation}
  \H_{c_0,\lambda}(f) = \H_{c_0}(f) + \frac12\lambda\A(f).
\end{equation}
In this article, we study the evolution equation given by the $L^2(\dd\mu)$ gradient flow of $\H_{c_0,\lambda}$, that is, for an immersion $f_0\colon\Sigma\to\R^3$, we study
\begin{equation}\label{eq:helfrich-iv-problem}
  \begin{cases}
    \partial_t f =-\big( \Delta H + |A^0|^2H - c_0\bigl( |A^0|^2-\frac12 H^2\bigr) - (\lambda+\frac12c_0^2) H \big)\nu&\text{on $[0,T)\times\Sigma$}\\
    f(0)=f_0&\text{on $\Sigma$}
  \end{cases}
\end{equation}
where $\Delta$ is the Laplace-Betrami operator on $(\Sigma,g_f)$, cf. \cite[Lemma~2.1]{mccoywheeler2016}. In the following, a family of immersions $f\colon[0,T)\times\Sigma\to\R^3$ satisfying \eqref{eq:helfrich-iv-problem} is called a $(c_0,\lambda)$-\emph{Helfrich flow} starting in $f_0$. Moreover, an immersion $f\colon\Sigma\to\R^3$ for which the right-hand side in \eqref{eq:helfrich-iv-problem} vanishes is referred to as a $(c_0,\lambda)$-\emph{Helfrich immersion}. We sometimes refer to the parameter $\lambda\geq 0$ as a \emph{local} area constraint. This is simply due to the local nature of the resulting evolution equation in \eqref{eq:helfrich-iv-problem}.

Already the case $c_0=\lambda=0$, i.e. the so-called \emph{Willmore flow}, is analytically challenging. First results are obtained in \cite{simonett2001} where Simonett shows well-posedness and stability of the round sphere, the absolute minimizer of the Willmore energy. In \cite{kuwertschaetzle2002}, Kuwert-Schätzle use parabolic interpolation techniques based on the Michael-Simon-Sobolev inequality in \cite{michaelsimon1973} to show a life-span theorem for the Willmore flow if no curvature concentration occurs. If $f_0$ is energetically close to the round sphere, in \cite{kuwertschaetzle2001}, they deduce global existence and convergence to the round sphere. Finally, in \cite{kuwertschaetzle2004}, Kuwert-Schätzle show that the Willmore flow starting in any spherical immersion with initial energy below $8\pi$ exists globally and converges to a round sphere --- a sharp energy threshold as analytically shown in \cite{blatt2009}. In \cite{rupp2023,rupp2024}, Rupp considers non-local variants of the Willmore flow where the volume or the isoperimetric ratio are kept fixed along the flow. Under different energy-thresholds for $f_0$, Rupp also deduces global existence and convergence to the round sphere in case of the volume preserving flow and to some Helfrich immersion (with $c_0=0$) for the flow with constant isoperimetric ratio.

The case where $c_0=0$ but $\lambda>0$ is already studied in \cite{mccoywheeler2016,blatt2019} where McCoy-Wheeler and Blatt find especially that the $(0,\lambda)$-Helfrich flow exhibits finite-time singularities. This behavior of the flow corresponds to the scaling behavior of the energy $\H_{0,\lambda}$. Furthermore, an analysis for the Helfrich flow similar to Simonett's initial work is undertaken in \cite{kohsakanagasawa2006} where local existence of the $(c_0,\lambda)$-Helfrich flow is shown and the center-manifold-method is employed to analyze the stability of the round sphere. As it turns out, the results already significantly depend on the choice of parameters. There are also some numerical studies of the Helfrich flow, e.g. cf. \cite{barrettgarckenuernberg2021}. 

This work's goal is giving a contribution to understanding the effects of a non-zero spontaneous curvature $c_0$ on the convergence behavior of the Helfrich flow. First, in the spirit of \cite{mccoywheeler2016,blatt2019}, we obtain the following ``negative result'' for a negative spontaneous curvature $c_0$.

\begin{theorem}\label{prop:neg-result}
  There exist universal constants $C>0$ and $\bar{\alpha}>0$ such that, if $c_0<0$, $\lambda\geq 0$ and $f_0\colon\S^2\to\R^3$ is an embedding with \eqref{eq:choice-of-orientation} and
  \begin{equation}\label{eq:sm-initial-en}
    \int_{\S^2}|A^0_{f_0}|^2\dd\mu_{f_0} \leq \exp\Big(-C\big((\H_{c_0,\lambda}(f_0))^2-(4\pi)^2\big)\Big)\bar{\alpha},
  \end{equation}
  then the maximal $(c_0,\lambda)$-Helfrich flow $f\colon[0,T)\times\S^2\to\R^3$ starting in $f_0$ satisfies 
  \begin{equation}
    T < \frac{4\big((\H_{c_0,\lambda}(f_0))^2-(4\pi)^2\big)}{\pi^2(2\lambda+c_0^2)^2}
  \end{equation}
  and $\A(f(t))\to 0$, $\H_{c_0,\lambda}(f(t))\to 4\pi$ and $\W(f(t))\to 4\pi$ for $t\nearrow T$. Moreover, there exist sequences $t_j\nearrow T$, $r_j\searrow 0$ and $x_j\in\R^3$ such that, up to reparametrization, $(f(t_j)-x_j)/r_j$ smoothly converges to a round sphere.
\end{theorem}
Note that \Cref{prop:neg-result} generalizes the behavior of round spheres to those surfaces which are energetically close to round spheres. Indeed, round spheres satisfy $|A^0|^2\equiv 0$, so in particular, \eqref{eq:sm-initial-en} is always satisfied if $f_0$ parametrizes a round sphere. 

While this behavior may be expected, we also find that a positive spontaneous curvature $c_0$ yields the following ``positive result'', i.e. has regularizing effects for the flow. This can be proved even without an energy assumption like \eqref{eq:sm-initial-en} requiring the initial surface to be (energetically) close to a round sphere.

\begin{theorem}\label{thm:main}
  Let $c_0>0$, $\lambda\geq 0$ and $f_0\colon\S^2\to\R^3$ be an immersion with \eqref{eq:choice-of-orientation} and
  \begin{equation}\label{eq:en-thresh}
    \H_{c_0,\lambda}(f_0)=\H_{c_0}(f_0)+\frac12\lambda\A(f_0)\leq \frac{2\lambda}{c_0^2+2\lambda}8\pi. 
  \end{equation}
  If $f\colon[0,T)\times\S^2\to\R^3$ is a maximal $(c_0,\lambda)$-Helfrich flow starting in $f_0$, then $T=\infty$. Moreover, the flow converges up to reparametrization to a Helfrich immersion for $t\to\infty$.
\end{theorem}

So the presence of a positive spontaneous curvature $c_0>0$ provides a certain regularizing effect in the sense that it prohibits short-time singularities.

\begin{remark}[Non-triviality of \eqref{eq:en-thresh}]\label{rem:en-thresh}
  If $c_0>0$ and $\lambda\geq 0$, one finds that, if $f_r\colon\S^2\to\R^3$ parametrizes a round sphere of radius $r=\frac{2c_0}{c_0^2+2\lambda}$, then 
  \begin{align}
    \beta_{c_0,\lambda}&\vcentcolon=\inf\{\H_{c_0}(f)+\frac12\lambda\A(f)\mid f\colon\S^2\to\R^3\text{ is an immersion}\}\\
    &\leq \H_{c_0}(f_r)+\frac12\lambda\A(f_r) = \frac{2\lambda}{c_0^2+2\lambda} 4\pi.
  \end{align}
\end{remark}

Particularly, one finds a commonality between the result in \Cref{thm:main} and the convergence result for the Willmore flow of spheres in \cite{kuwertschaetzle2004}. Namely, Kuwert-Schätzle show global existence and convergence of the Willmore flow of spheres if $\W(f_0)\leq 8\pi$. Note that the energy threshold $8\pi$ is exactly twice the absolute minimum of the Willmore energy. Similarly, the energy threshold in \eqref{eq:en-thresh} is twice the energy of the round sphere minimizing $\H_{c_0,\lambda}$ among round spheres, cf. \Cref{rem:en-thresh}.

\begin{remark}
  One can fully classify the behavior of the Helfrich flow starting in round spheres. Let $f\colon\S^2\to\R^3$ be any parametrization of the round sphere with radius $1$ centered in $0$. Furthermore, fix $x_0\in\R^3$. Then, for $r>0$, consider the immersions $f_r\colon\S^2\to\R^3$ with $f_r(x)=x_0+rf(x)$. In the orientation of $\S^2$, one has $\nu_{f_r}(x)=-f(x)$ and $H_{f_r}\equiv\frac{2}{r}$, $|A^0_{f_r}|\equiv 0$. Particularly, if $r\colon[0,T)\to(0,\infty)$ is smooth, writing $f(t,x)=f_{r(t)}(x)$, one finds that $f$ is a $(c_0,\lambda)$-Helfrich flow if and only if
  \begin{equation}\label{eq:spheres-ode}
    \frac{\dd}{\dd t} r(t) = \frac{c_0}{r(t)} \Bigl(\frac{2}{r(t)}-c_0\Bigr) - \frac{2}{r(t)}\lambda\quad\text{on $[0,T)$}.
  \end{equation}
  An analysis of this autonomous, one-dimensional ordinary differential equation gives the following. Let $c_0\in\R\setminus\{0\}$, $\lambda\geq 0$, $r_0>0$ and $f_0\colon\S^2\to\R^3$ parametrize a round sphere of radius $r_0$ centered at $x_0\in\R^3$. Denote by $f\colon[0,T)\times\S^2\to\R^3$, $f(t,x)=x_0+r(t)\frac{f_0(x)-x_0}{r_0}$ the maximal $(c_0,\lambda)$-Helfrich flow starting in $f_0$ where $r$ solves \eqref{eq:spheres-ode}. If $c_0<0$, then $T<\infty$ and $r(t)\searrow 0$ for $t\nearrow T$. If however $c_0>0$, then $T=\infty$ and
  \begin{equation}
    r(t)\to r^*=\frac{2c_0}{c_0^2+2\lambda}\in(0,\infty)\quad\text{for $t\to\infty$}.
  \end{equation}
\end{remark}

\subsection{Strategy of proof and outline}

When expanding the square in $\int_{\S^2} (H-c_0)^2\dd\mu$, it already becomes apparent that the sign of $\int_{\S^2} H\dd\mu$ plays a role when investigating the influence of a non-zero spontaneous curvature $c_0$. Since any a-priori information on the sign of the average mean curvature is in general hard to obtain, in \Cref{prop:neg-result}, we work in the setting of surfaces whose Willmore energy is sufficiently close to that of the round sphere which has positive average mean curvature. Indeed, in \Cref{prop:pos-av-mc-for-sm-en}, using the compactness-theory for weak Sobolev immersions by \cite{riviere2014,kuwertli2012,mondinoriviere2014,mondinoscharrer2020}, we find that being energetically close to a round sphere suffices to control the sign of $\int_{\S^2}H\dd\mu$. Then, we employ the fact that the Helfrich flow preserves smallness of the Willmore energy in finite time. Finally, with this a-priori knowledge, one can apply a similar scaling-argument as in \cite{blatt2019} also for the energy $\H_{c_0,\lambda}$ with $c_0<0$ to deduce the existence of finite-time singularities.

For the proof the \Cref{thm:main}, in the spirit of \cite{kuwertschaetzle2002}, we first compute the local decay of the $L^2$-norm of the second fundamental form in \Cref{sec:loc-en}. Then, in \Cref{sec:blow-up}, we show the existence of a blow-up and suitable concentration limit along the Helfrich flow, especially proving a life-span theorem along the way. Note that, since the Helfrich energy (and thus the Helfrich flow) are not scaling-invariant, it is especially important to keep track of the dependence of all estimates on the parameters $c_0$ and $\lambda$. Finally, in \Cref{sec:loja}, a \L ojasiewicz-Simon gradient inequality is used to show asymptotic stability of the Helfrich flow if the concentration limit in the blow-up procedure is compact. Here, it is important to understand why finite-time curvature concentration by ``scaling'' (as one observes in the case $c_0<0$ in \Cref{prop:neg-result}) cannot occur under the assumption $c_0>0$ and the energy threshold \eqref{eq:en-thresh}. This is a delicate argument carried out in \Cref{subsec:rneq0}. 

\begin{remark}
  Excluding such curvature concentration by ``scaling'' is significantly easier if one additionally assumes a $4\pi$-energy constraint in \eqref{eq:en-thresh}, i.e. additionally requiring $\H_{c_0,\lambda}(f_0)\leq 4\pi$: If $\A(f(t))\to 0$ for $t\nearrow T$, then $\H_{c_0,\lambda}(f(t))\leq 4\pi-\beta$ contradicts $|\W(f(t))-\H_{c_0,\lambda}(f(t))|\to 0$ and $\W(f(t))\geq 4\pi$. Here we use that, for a sequence of immersions $f_j\colon\Sigma\to\R^3$ with $\W(f_j)\leq C$ and $\A(f_j)\to 0$, 
  \begin{equation}
    |\W(f_j)-\H_{c_0,\lambda}(f_j)|\leq \frac12|c_0|\int_{\Sigma}|H|\dd\mu + (\frac14c_0^2+\lambda)\A(f_j) \leq \frac12|c_0|\sqrt{4C}\sqrt{\A(f_j)} + o(1) \to 0
  \end{equation}
  for $j\to\infty$.
\end{remark}

Finally, with all these preliminary results, in \Cref{sec:proof-main}, we complete the proofs of \Cref{prop:neg-result,thm:main}.

\section{Geometric preliminaries}

\subsection{Curvature identities and estimates}

For more details on the geometric preliminaries, cf. \cite[Section~1.1]{kuwertschaetzle2012} and \cite[Section~2]{mccoywheeler2013}. Consider an immersion $f\colon\Sigma\to\R^3$ of a closed, oriented surface $\Sigma$. Let $\varphi\colon\Sigma\to\R$ be smooth. For a Borel-set $K\subseteq\R^3$, we often write $\int_{K} \varphi\dd\mu \vcentcolon= \int_{f^{-1}(K)} \varphi\dd\mu$. One computes
\begin{equation}\label{eq:a-a0-h}
  |A|^2=|A^0|^2+\frac12 H^2.
\end{equation}
Note that, as a consequence of the Gauss-Bonnet Theorem, using that the \emph{Gauss curvature} $K$ can be written as
\begin{equation}\label{eq:Gauss-curv}
  K=\frac14H^2 - \frac12|A^0|^2,
\end{equation}
one has for $\W_0(f)=\int_{\Sigma}|A^0|^2\dd \mu$
\begin{equation}\label{eq:int-a0sq}
  \W_0(f)=2\W(f)-8\pi(1-g)
\end{equation}
where $g\in\N_0$ denotes the genus of $\Sigma$. Particularly, combining \eqref{eq:a-a0-h} and \eqref{eq:int-a0sq},
\begin{equation}\label{eq:int-asq}
  \int_{\Sigma} |A|^2\dd\mu = \W_0(f) + 2\W(f) = 4(\W(f)-2\pi)+8\pi g.
\end{equation}
In the coordinates of a local orthonormal frame $e_1,e_2$, by the definition of $H$, one has $H=A_{ii}$ and thus $|\nabla^m H|\lesssim
|\nabla^m A|$. Using $A^0_{ij}=A_{ij}-\frac12\delta_{ij}H$, one also finds $|\nabla^m A^0|\lesssim |\nabla^m A|$. As explained in \cite[Equations~(6) and (7)]{mccoywheeler2013}, as a consequence of the Codazzi equations, also the reverse holds. To summarize,
\begin{equation}\label{eq:cod-main}
  |\nabla^m H|\lesssim|\nabla^m A|\lesssim |\nabla^mA^0| \lesssim |\nabla^mA|.
\end{equation}   
\begin{remark}\label{rem:neg-c0-4pi}
  If $c_0<0$ and $f\colon\Sigma\to\R^3$ is an Alexandrov immersion, then one finds $\H_{c_0}(f)>4\pi$, cf. \cite[Lemma~6.1 and Theorem~1.5]{ruppscharrer2023}. Particularly, this holds for embeddings, so by the Li-Yau inequality in \cite{liyau1982}, for immersions $f$ with $\W(f)<8\pi$.
\end{remark}

\begin{lemma}[Controlling the Willmore energy]
  Let $f\colon\Sigma\to\R^3$ be an immersion and $c_0\in\R$ as well as $\lambda>0$. Then
  \begin{equation}\label{eq:est-Will-energy}
    \W(f) \leq \frac{2\lambda+c_0^2}{2\lambda} \Bigl( \H_{c_0}(f) + \frac12\lambda \A(f) \Bigr) = \frac{2\lambda+c_0^2}{2\lambda} \cdot \H_{c_0,\lambda}(f).
  \end{equation}
\end{lemma}
\begin{proof}
  W.l.o.g. $c_0\neq 0$. Using $\W(f)=\H_{c_0}(f) + \frac12c_0\int_{\Sigma}H\dd\mu - \frac14c_0^2\A(f)$ and the Cauchy-Schwarz estimate $\frac12c_0\int_{\Sigma}H\dd\mu\leq \sqrt{\W(f)}\sqrt{c_0^2\A(f)}$, one finds
  \begin{align}
    \W(f)\leq \H_{c_0}(f) + \Bigl(\varepsilon-\frac14\Bigr)c_0^2\A(f)+ \frac{1}{4\varepsilon}\W(f)
  \end{align}
  for any $\varepsilon>0$, using the Peter-Paul inequality. Choosing $\varepsilon=\frac12\lambda/c_0^2 + \frac14$, one obtains
  \begin{equation}
    \Bigl(1-\frac{c_0^2}{2\lambda+c_0^2}\Bigr)\W(f)\leq \H_{c_0}(f)+\frac12\lambda\A(f),
  \end{equation}
  that is, \eqref{eq:est-Will-energy}.
\end{proof}

\subsection{Evolution equations}
First, recall the following evolutions of the relevant geometric quantities.
\begin{lemma}[{\cite[Lemma~2.3]{rupp2024}}] For a family of immersions $f\colon[0,T)\times \Sigma \to\R^3$ with normal speed $\partial_t f=\vcentcolon \xi\nu$, one has in the coordinates of a local orthonormal frame ${e_1,e_2}$ 
  \begin{align}
    \partial_t \dd\mu &= -H\xi\dd\mu\label{eq:ev-mu}\\
    \partial_t H &= \Delta\xi + |A|^2\xi=\Delta\xi + |A^0|^2\xi + \frac{1}{2}H^2\xi\label{eq:ev-H}\\
    \partial_t (H\dd\mu) &= \Delta\xi\dd\mu +  (|A^0|^2 - \frac{1}{2}H^2)\xi\dd\mu  \label{eq:ev-Hdmu}\\
    (\partial_t A)(e_i,e_j)&=\nabla_{i,j}^2\xi - A_{ik}A_{kj}\xi.
  \end{align}
\end{lemma}

As in \cite[Proposition~2.4]{rupp2024}, these evolutions immediately yield

\begin{proposition}\label{prop:ev-eq}
  Let $f\colon\Sigma\to\R^3$ be an immersion and $\varphi\colon\Sigma\to\R^3$. One has the following first-variation identities.
  \begin{align}
    \W_0'(f)(\varphi)&=\int_{\Sigma} \langle [\Delta H+|A^0|^2H]\nu,\varphi\rangle\dd\mu =\vcentcolon \langle \nabla\W_0(f),\varphi\rangle_{L^2(\dd \mu)},\\
    \A'(f)(\varphi)&=-\int_{\Sigma} \langle H\nu,\varphi\rangle\dd\mu =\vcentcolon \langle \nabla\A(f),\varphi\rangle_{L^2(\dd \mu)},\\
    \V'(f)(\varphi)&=-\int_{\Sigma} \langle \nu,\varphi\rangle\dd\mu =\vcentcolon \langle \nabla\V(f),\varphi\rangle_{L^2(\dd \mu)}
  \end{align}
  and finally $\H_{c_0}'(f)(\varphi) = \langle \nabla\H_{c_0}(f),\varphi\rangle_{L^2(\dd \mu)}$ with
  \begin{align}
    2\nabla \H_{c_0}(f)&=\nabla\W_0(f)-c_0|A^0|^2\nu + \frac12c_0H(H-c_0)\nu\\
    &= [\Delta H + |A^0|^2(H-c_0) + \frac12 c_0 H(H-c_0)]\nu. 
  \end{align}
\end{proposition}

\subsection{Fundamental properties of the $(c_0,\lambda)$-Helfrich flow}

Given $c_0\in\R$ and $\lambda\in\R$, recall that a family of immersions $f\colon[0,T)\times \Sigma\to\R^3$ satisfying
\begin{align}\label{eq:flow-eq}
  \partial_t f &= -2\nabla\H_{c_0}(f) - \lambda\nabla\A(f) = \xi\nu  \quad\text{with}\\
  \xi &=-\big( \Delta H + |A^0|^2H - c_0\bigl( |A^0|^2-\frac12 H^2\bigr) - (\lambda+\frac12c_0^2) H \big)
\end{align}
is referred to as a $(c_0,\lambda)$-Helfrich flow with initial datum $f_0=f(0)$.
\begin{remark}[Energy decay]\label{rem:en-dec}
  If $f\colon[0,T)\times\Sigma\to\R^3$ is a $(c_0,\lambda)$-Helfrich flow, then
  \begin{equation}
    \frac{\dd}{\dd t}\bigl(2\H_{c_0}(f(t))+\lambda\A(f(t))\bigr) =  \int_{\Sigma} \langle 2\nabla \H_{c_0}(f)+\lambda\nabla\A(f),\partial_tf\rangle\dd \mu  =-\int_{\Sigma} |\partial_tf|^2\dd \mu \leq 0.
  \end{equation}
  Particularly, if $\H_{c_0,\lambda}(f_0)=\H_{c_0}(f_0)+\frac12\lambda\A(f_0)\leq K$, then 
  \begin{equation}\label{eq:en-dec}
    \H_{c_0,\lambda}(f(t))=\H_{c_0}(f(t))+\frac12\lambda\A(f(t)) \leq K\quad\text{for all $0\leq t<T$}.
  \end{equation} 
\end{remark}
\begin{remark}[Bound on the Willmore energy]\label{rem:est-will-en-along-flow}
  If $f\colon[0,T)\times\Sigma\to\R^3$ is a $(c_0,\lambda)$-Helfrich flow with $c_0,\lambda>0$ and $\H_{c_0,\lambda}(f_0) < \frac{2\lambda}{c_0^2+2\lambda}8\pi$, then
  \begin{equation}\label{eq:will-below-8pi}
    \W(f(t))\leq \frac{2\lambda+c_0^2}{2\lambda}\H_{c_0,\lambda}(f(t))\leq \frac{2\lambda+c_0^2}{2\lambda}\H_{c_0,\lambda}(f_0)\leq 8\pi-\beta
  \end{equation} 
  for some $\beta>0$, using \eqref{eq:est-Will-energy} and \Cref{rem:en-dec}.
\end{remark}

\begin{remark}[Well-posedness of \eqref{eq:helfrich-iv-problem}]\label{rem:well-posedness}
  With similar computations as on \cite[pp.~27~--~28]{kuwertschaetzle2012}, as \eqref{eq:flow-eq} differs from the Willmore flow only in lower-order terms, one finds that \cite{mantegazzamartinazzi2012} applies to the analog of \cite[Equation~(3.1.6)]{kuwertschaetzle2012} to obtain well-posedness of the initial value problem \eqref{eq:helfrich-iv-problem}. Short-time existence with weaker initial regularity in some little-Hölder space is proved in \cite[Theorem~2.1]{kohsakanagasawa2006}.
\end{remark}

The following scaling behavior of the flow and the underlying parameters $c_0$ and $\lambda$ is fundamental for a blow-up analysis. Note that, in contrast to the Willmore energy, $\H_{c_0}$ is not scaling-invariant --- upon scaling, the parameter $c_0$ changes. 

\begin{lemma}\label{lem:par-scal}
  Let $c_0,\lambda\in\R$ and $r>0$, $x\in\R^3$. If $f\colon[0,T)\times\Sigma\to\R^3$ is a $(c_0,\lambda)$-Helfrich flow, then its parabolic rescaling $\widetilde{f}\colon[0,T/r^4)\times\Sigma\to\R^3$ with 
  \begin{equation}
    \widetilde{f}(t,p)=\frac{1}{r}\bigl(f(r^4t,p)-x\bigr)
  \end{equation}
  is an $(rc_0,r^2\lambda)$-Helfrich flow.
\end{lemma}
\begin{proof}
  The claim is a standard computation using $r^3\nabla\W_0(f(r^4t)) = \nabla\W_0(\widetilde{f}(t))$ and $r H_{f(r^4t)} = H_{\widetilde{f}(t)}$ as well as $r|A^0_{f(r^4t)}| = |A^0_{\widetilde{f}(t)}|$.
\end{proof}

\section{Finite-time singularities for negative spontaneous curvature}

As argued in the introduction, for \Cref{prop:neg-result}, we need to control the sign of $\int_{\S^2}H\dd\mu$ along a Helfrich flow. Since it is generally easier to control energies along gradient flows, we first make the following observation. 

\begin{proposition}\label{prop:pos-av-mc-for-sm-en}
  There exists a constant $0<\alpha_0<8\pi$ such that, for any smooth immersion $f\colon\S^2\to\R^3$ with $\V(f)=-\frac13\int_{\S^2}\langle f,\nu\rangle\dd\mu > 0$ and
  \begin{equation}\label{eq:sm-en}
    \int_{\S^2} |A^0|^2\dd\mu < \alpha_0,
  \end{equation}
  one has $\int_{\S^2} H \dd\mu>0$.
\end{proposition}
\begin{remark}
  In the above, we require $0<\alpha_0<8\pi$ to ensure that \eqref{eq:sm-en} yields $\W(f)<8\pi$, using \eqref{eq:int-a0sq}, which is useful later on. Note that, by \cite[Corollary~8.4]{scharrerwest2024}, for axi-symmetric spherical immersions in $\R^3$, one can take $\alpha_0=4\pi$ in \eqref{eq:sm-en}.
\end{remark}
\begin{proof}[Proof of \Cref{prop:pos-av-mc-for-sm-en}.]
  For the sake of contradiction, suppose that there is a sequence of immersions $f_j\colon\S^2\to\R^3$ with $\int_{\S^2}|A^0_j|^2\dd\mu_j\to 0$ and  
  \begin{equation}\label{eq:mc-sm-en-1}
    \V(f_j)=-\frac{1}{3}\int_{\S^2}\langle f_j,\nu_j\rangle\dd\mu_j > 0\text{ and }\int_{\S^2}H_j\dd\mu_j\leq 0\quad \text{for all $j\in\N$}.
  \end{equation}
  As the energy in \eqref{eq:sm-en} is scaling-invariant, we may w.l.o.g. suppose that $\A(f_j)=1$ for all $j\in\N$. Arguing as in \cite[Corollary~2.32, using Theorems~1.17~and~1.9]{riviere2016}, we can furthermore w.l.o.g. assume that each $f_j$ is parametrized as a conformal immersion. That is, in any conformal coordinates $x=(x^1,x^2)$ with respect to the standard metric on $\S^2$, one has
  \begin{equation}
    |\partial_{x^1}f_j|^2=|\partial_{x^2}f_j|^2\quad\text{and}\quad \langle \partial_{x^1}f_j,\partial_{x^2}f_j\rangle \equiv 0.
  \end{equation} 
  Using \eqref{eq:a-a0-h}, $\W(f_j)\to 4\pi$ and $\A(f_j)\equiv 1$, one finds
  \begin{equation}
    \limsup_{j\to\infty} \int_{\S^2} 1+|A_j|^2\dd\mu_j = 1 + 8\pi <\infty.
  \end{equation}
  Moreover, by Simon's lower diameter estimate in \cite[Lemma~1.1]{simon1993}, using $\W(f_j)\to4\pi$ and $\A(f_j)\equiv 1$, one finds $\inf_{j\in\N}\mathrm{diam}(f_j(\S^2))>0$. Therefore, \cite[Theorem~1.5]{mondinoriviere2014} (also cf. \cite[Theorem~1.6]{mondinoscharrer2020}) yields that one has the following after passing to a subsequence which, for the sake of convenience, we do not relabel. There exists a family of bilipschitz homeomorphisms $\Psi_j$ of $\S^2$, $N\in\N$, sequences $\Phi_j^1, \ldots , \Phi_j^N$ of positive conformal diffeomorphisms of $\S^2$, and weak (possibly branched) immersions with finite total curvature $f_{\infty}^1 , \ldots , f_{\infty}^N$ (cf. \cite[Definition~1.3]{mondinoscharrer2020}) such that the following is satisfied. For some $f_{\infty}\in W^{1,\infty}(\S^2,\R^3)$, one has
  \begin{equation}
    f_j\circ\Psi_j\to f_{\infty}\quad\text{in $C^0(\S^2,\R^3)$}
  \end{equation}
  and, for points $b^{i,j}\in\S^2$ where $1\leq j\leq N_i$, $1\leq i\leq N$,
  \begin{equation}
    f_j\circ\Phi_j^i\rightharpoonup f_{\infty}^i\quad\text{weakly as $j\to\infty$ in $W^{2,2}_{\mathrm{loc}}(\S^2\setminus\{b^{i,1},\dots,b^{i,N_i}\},\R^3)$}
  \end{equation}
  for $i=1,\dots,N$. Arguing as in \cite[Theorem~3.3]{mondinoscharrer2020} with $T_j=(f_j,f_j)$ and \cite[Equation~(2.13)]{mondinoscharrer2020},
  \begin{equation}\label{eq:mc-sm-en-4}
    4\pi N \leq \sum_{i=1}^N \int_{\S^2} H_{f_{\infty}^i}^2\dd\mu_{f_{\infty}^i} \leq \liminf_{j\to\infty} \int_{\S^2}H_{f_j}^2\dd\mu_j = 4\pi,
  \end{equation}
  so $N=1$ and as $\W(f_{\infty}^1)<8\pi$, $f_{\infty}^1$ is an embedding. Therefore, by \cite[Theorem~3.1]{kuwertli2012}, $f_{\infty}^1$ is a weak conformal immersion without any branch points. Now, as in \cite[Theorem~3.3]{mondinoscharrer2020}, using \eqref{eq:mc-sm-en-1}, one also finds
  \begin{align}
    \int_{\S^2}1\dd\mu_{f_{\infty}^1}=\lim_{j\to\infty}\A(f_j)&=1,\quad \V(f_{\infty})=-\frac13 \int_{\S^2} \langle f_{\infty}^1,\nu_{f_{\infty}^1}\rangle\dd\mu_{f_{\infty}^1} \geq 0,\\
    \int_{\S^2}H_{f_{\infty}^1}\dd\mu_{f_{\infty}^1} &= \lim_{j\to\infty}\int_{\S^2}H_j\dd\mu_j\leq 0.\label{eq:mc-sm-en-3}
  \end{align}
  By \eqref{eq:mc-sm-en-4}, $\W(f_{\infty}^1)=4\pi$ and thus $f_{\infty}^1$ is the absolute minimizer of the Willmore energy. Since $f_{\infty}^1$ further does not have any branch points, \cite[Theorem~4.3]{mondinoscharrer2020} yields $f_{\infty}^1\in C^{\infty}(\S^2,\R^3)$. By classical theory in \cite{willmore1965}, also cf. \cite[Theorem~7.2.2]{willmore1993} or \cite[Proposition~1.1.1]{kuwertschaetzle2012}, $f_{\infty}^1$ parametrizes a round sphere. This however contradicts \eqref{eq:mc-sm-en-3}!
\end{proof}

Now we see how a positive average mean curvature and negative spontaneous curvature come together to result in finite time singularities for the Helfrich flow.

\begin{lemma}[A bound on the maximal existence time via a scaling argument]\label{lem:max-ex-time}
  Let $f\colon[0,\bar{t})\times\S^2\to\R^3$ be a $(c_0,\lambda)$-Helfrich flow with $\int_{\S^2}H\dd\mu \geq 0$ on $[0,\bar{t})$ where $c_0<0$ and $\lambda\geq 0$, then $\bar{t}<\infty$. Moreover, if $\W(f(t))<8\pi$ for all $0\leq t<\bar{t}$, then
  \begin{equation}
    \bar{t} < \frac{4\big((\H_{c_0,\lambda}(f_0))^2-(4\pi)^2\big)}{\pi^2(2\lambda+c_0^2)^2}.
  \end{equation}
\end{lemma}
\begin{proof}
  Fix $0\leq t<T$ and $x_t\in f(t,\S^2)$. Then one finds as in \Cref{lem:par-scal} 
  \begin{align}
    \frac{\dd}{\dd r} 2\H_{c_0,\lambda}(r(f-x_t))\Big|_{r=1} &= \frac{\dd}{\dd r} 2\H_{rc_0,r^2\lambda}(f)\Big|_{r=1} = (2\lambda+c_0^2)\A(f)-c_0\int_{\S^2}H\dd\mu\\
    &\geq (2\lambda+c_0^2)\A(f).
  \end{align}
  On the other hand, using \Cref{rem:en-dec}, Topping's improvement of Simons's diameter estimate in \cite[Lemma~1]{topping1998}, Cauchy-Schwarz and $\W(f)\leq \H_{c_0}(f)-\frac14c_0^2\A(f)$,
  \begin{align}
    \frac{\dd}{\dd r}2\H_{c_0,\lambda}(r(f-x_t))\Big|_{r=1} &= -\int_{\S^2} \langle \partial_t f,f-x_t\rangle\dd \mu \leq \mathrm{\diam}(f(t)) \int_{\S^2}|\partial_tf|\dd\mu \\
    &\leq \frac{2}{\pi} \A(f(t)) \sqrt{\W(f(t))} \Big(\int_{\S^2}|\partial_tf|^2\dd\mu\Big)^{\frac12} \\
    &\leq \frac{2}{\pi} \sqrt{\H_{c_0}(f(t))-\frac{1}{4}c_0^2\A(f(t))}\Big(-2\frac{\dd}{\dd t}\H_{c_0,\lambda}(f(t))\Big)^{\frac{1}{2}}\cdot\A(f(t)).
  \end{align}
  That is, estimating $\H_{c_0}(f(t))-\frac{1}{4}c_0^2\A(f(t))\leq\H_{c_0,\lambda}(f(t))$,
  \begin{equation}
    \frac{\dd}{\dd t}\H_{c_0,\lambda}(f(t)) \leq -\frac{\pi^2}{8} \frac{(2\lambda+c_0^2)^2}{\H_{c_0,\lambda}(f(t))},
  \end{equation}
  and using $\H_{c_0,\lambda}(f(t))>4\pi$ in case that $\W(f(t))<8\pi$, cf. \Cref{rem:neg-c0-4pi}, the claim follows.
\end{proof}

With \Cref{prop:pos-av-mc-for-sm-en,lem:max-ex-time} in mind, one can deduce the existence of singularities in finite time if \eqref{eq:sm-en} is satisfied along the flow. Therefore, we make the following estimate.

\begin{lemma}\label{lem:sm-en}
  With $\alpha_0$ as in \Cref{prop:pos-av-mc-for-sm-en}, there exists $0<\bar{\alpha}<\alpha_0$ with the following property. Let $c_0,\lambda\in\R$ not both vanish and consider $0<\alpha_1\leq \bar{\alpha}$. If $f\colon[0,T)\times\S^2\to\R^3$ is a $(c_0,\lambda)$-Helfrich flow such that, for some $0<\varepsilon < 1$,
  \begin{equation}\label{eq:sm-en-alpha1}
    \W_0(f_0)=\int_{\S^2}|A^0|^2\dd\mu\Big|_{t=0} \leq (1-\varepsilon)\alpha_1,
  \end{equation}
  then, writing $\overline{t}\vcentcolon= \sup\{t\in[0,T)\mid \int_{\S^2}|A^0|^2\dd\mu\leq\alpha_1\text{ on $[0,t]$}\}$, one has $\overline{t}=T$ or 
  \begin{equation}
    \overline{t} \geq -\frac{\log(1-\varepsilon)}{C(\lambda^2+c_0^4)}
  \end{equation}
  for some universal constant $C>0$. Moreover, for $0\leq t<\bar{t}$,
  \begin{align}
    \int_{0}^{t} \int_{\S^2}\big(|\nabla^2A|^2+|\nabla A|^2|A|^2+|A^0|^2|A|^4\big)\dd\mu\dd\tau &\leq \big((1-\varepsilon)+C(\lambda^2+c_0^4)t\big)\alpha_1,\\
    \int_0^t \|A^0\|_{L^{\infty}}^4\dd\tau &\leq  C \big((1-\varepsilon)+C(\lambda^2+c_0^4)t\big) \alpha_1^2.\label{eq:sm-en-linfty-estimate}
  \end{align}
\end{lemma}
\begin{proof}
  By continuity, we have $\overline{t}>0$. Consider any $t\in[0,\overline{t})$ so that, particularly, $\W_0(f(t))\leq \alpha_1$. For $0<\eta<1$ to be chosen later, using \Cref{prop:ev-eq} and \eqref{eq:flow-eq},
  \begin{align}
    \frac{\dd}{\dd t} \int_{\S^2}|A^0|^2\dd\mu &\ + \int_{\S^2}|\nabla\W_0(f)|^2\dd\mu\\
    &= \int_{\S^2} \big(c_0|A^0|^2-\frac12c_0H^2+\frac12(c_0^2+\lambda)H\big)\langle\nabla\W_0(f),\nu\rangle \dd\mu\\
    &\leq \frac12 \int_{\S^2} |\nabla\W_0(f)|^2\dd\mu + \frac12 c_0^2\int_{\S^2}|A^0|^4\dd\mu \\
    &\qquad + \frac12\int_{\S^2} \big((c_0^2+\lambda)H-c_0H^2\big)\big(\Delta H+|A^0|^2H\big)\dd\mu.
  \end{align}
  Choosing $\bar{\alpha}$ sufficiently small, applying \cite[Proposition~2.6]{kuwertschaetzle2001} with $\W_0(f(t))\leq\alpha_1$ as in \cite[Equation~(36)]{mccoywheeler2016}, for some universal constant $\widetilde{\delta_0}>0$, 
  \begin{equation}
    \widetilde{\delta_0} \int_{\S^2}\big(|\nabla^2A|^2+|A|^2|\nabla A|^2+|A^0|^2|A|^4\big)\dd\mu \leq \frac12 \int_{\S^2}|\nabla\W_0(f)|^2\dd\mu.
  \end{equation}
  Therefore, one finds
  \begin{align}
    \frac{\dd}{\dd t} &\ \int_{\S^2}|A^0|^2\dd\mu  + \widetilde{\delta_0} \int_{\S^2}\big(|\nabla^2A|^2+|A|^2|\nabla A|^2+|A^0|^2|A|^4\big)\dd\mu\\
    &\leq \frac12 c_0^2\int_{\S^2}|A^0|^4\dd\mu + \frac12\int_{\S^2} \big((c_0^2+\lambda)H-c_0H^2\big)\big(\Delta H+|A^0|^2H\big)\dd\mu.\label{eq:lem-sm-en-1}
  \end{align}
  Using the Simon's identity as in \cite[Equation~(68)]{kuwertschaetzle2001} and integrating by parts, one finds
  \begin{equation}\label{eq:lem-sm-en-2}
    \int_{\S^2}H\big(\Delta H+|A^0|^2H\big)\dd\mu \leq -2\int_{\S^2}|\nabla A^0|^2\dd\mu + C \int_{\S^2}|A^0|^4\dd\mu.
  \end{equation} 
  Moreover, using \eqref{eq:cod-main} and integrating by parts, for some $0<\eta<1$ and a universal constant $C>0$ changing from line to line,
  \begin{align}
    \Big|\frac12c_0\int_{\S^2} H^2\Delta H\dd\mu\Big| &= \Big|c_0 \int_{\S^2} H |\nabla H|^2\dd\mu\Big|\leq C |c_0| \int_{\S^2}|H| |\nabla A^0|^2\dd\mu \\
    &\leq \eta \int_{\S^2}|\nabla A|^2|A|^2\dd\mu + \frac{C}{\eta} c_0^2\int_{\S^2}|\nabla A^0|^2\dd\mu.\label{eq:lem-sm-en-3}
  \end{align}
  Furthermore, integrating by parts, one finds
  \begin{align}
    \frac{C}{\eta}c_0^2\int_{\S^2} |\nabla A^0|^2\dd\mu = -\frac{C}{\eta}c_0^2\int_{\S^2} \langle A^0,\Delta A^0\rangle_{g_f}\dd\mu \leq \eta \int_{\S^2} |\nabla^2 A|^2\dd\mu + \frac{C}{\eta^2}c_0^4\int_{\S^2}|A^0|^2\dd\mu.
  \end{align}
  So, suitably choosing $\eta\in(0,1)$ and combining this estimate with \eqref{eq:lem-sm-en-3} yields for some universal constant $C>0$
  \begin{equation}\label{eq:lem-sm-en-4}
    \Big|\frac12c_0\int_{\S^2} H^2\Delta H\dd\mu\Big| \leq \frac{\widetilde{\delta_0}}{4} \int_{\S^2} \big(|\nabla^2A|^2+|A|^2|\nabla A|^2\big)\dd\mu + C c_0^4 \int_{\S^2}|A^0|^2\dd\mu.
  \end{equation}
  Altogether, with \eqref{eq:lem-sm-en-1}, \eqref{eq:lem-sm-en-2} and \eqref{eq:lem-sm-en-4}, for $\delta_0=\frac12\widetilde{\delta_0}$ and a universal constant $C>0$ changing from line to line, using \eqref{eq:cod-main} and Young's inequality,
  \begin{align}
    \frac{\dd}{\dd t} \int_{\S^2}|A^0|^2\dd\mu&+\frac32\delta_0\int_{\S^2}\big(|\nabla^2 A|^2+|\nabla A|^2|A|^2+|A^0|^2|A|^4\big)\dd\mu \leq C c_0^4 \int_{\S^2} |A^0|^2\dd\mu\\
    &\qquad + C(\lambda+c_0^2)\int_{\S^2}|A^0|^4\dd\mu -\frac12c_0\int_{\S^2}|A^0|^2H^3\dd\mu\\
    &\quad \leq \frac{1}{4}\delta_0\int_{\S^2}|A^0|^2|A|^4\dd\mu + C c_0^2\int_{\S^2}|A^0|^2H^2\dd\mu + C(\lambda^2+c_0^4)\int_{\S^2}|A^0|^2\dd\mu\\
    &\quad \leq \frac{1}{2}\delta_0\int_{\S^2}|A^0|^2|A|^4\dd\mu  + C(\lambda^2+c_0^4)\int_{\S^2}|A^0|^2\dd\mu.
  \end{align}
  Now the Gronwall lemma implies for $t\in[0,\overline{t})$
  \begin{equation}
    \int_{\S^2}|A^0|^2\dd\mu\Big|_{t} \leq (1-\varepsilon)\alpha_1 \exp(C(\lambda^2+c_0^4)t)
  \end{equation}
  and thus the first part of the claim follows. Equation \eqref{eq:sm-en-linfty-estimate} can be concluded from the above estimates, using $\int_{\S^2}|A^0|^2\dd\mu \leq \alpha_1$ for $0\leq t<\bar{t}$ and the interpolation estimate in \cite[Lemma~2.8]{kuwertschaetzle2001}, also cf. \cite[Theorem~2.9]{kuwertschaetzle2001} with $\rho\to\infty$.
\end{proof}

\begin{corollary}\label{cor:fin-time-sing}
  Let $c_0<0$ and $\lambda\geq 0$ and consider an embedding $f_0\colon\S^2\to\R^3$ with
  \begin{equation}
    \int_{\S^2}|A^0_{f_0}|^2\dd\mu_{f_0} = 2(\W(f_0)-4\pi) \leq \exp\Big(-\frac{4C}{\pi^2}\big((\H_{c_0,\lambda}(f_0))^2-(4\pi)^2\big)\Big)\bar{\alpha}
  \end{equation}
  where $C>0$, $\bar{\alpha}>0$ are the universal constants in \Cref{lem:sm-en}. Then the maximal $(c_0,\lambda)$-Helfrich flow $f\colon[0,T)\times\S^2\to\R^3$ starting in $f_0$ satisfies  
  \begin{equation}
    T < \frac{4\big((\H_{c_0,\lambda}(f_0))^2-(4\pi)^2\big)}{\pi^2(2\lambda+c_0^2)^2}.
  \end{equation}
\end{corollary}
\begin{proof}
  By \Cref{lem:sm-en}, setting
  \begin{equation}
    t_0\vcentcolon=\frac{4\big((\H_{c_0,\lambda}(f_0))^2-(4\pi)^2\big)}{\pi^2(\lambda^2+c_0^4)},
  \end{equation}
  on $[0,\min\{t_0,T\})$, one has that
  \begin{equation}
    \int_{\S^2}|A^0|^2\dd\mu\Big|_{t} \leq \bar{\alpha}.
  \end{equation}
  Thus, using $\bar{\alpha}<\alpha_0$, \Cref{prop:pos-av-mc-for-sm-en} and \Cref{lem:max-ex-time} yield
  \begin{equation}
    \min\{t_0,T\} < \frac{4\big((\H_{c_0,\lambda}(f_0))^2-(4\pi)^2\big)}{\pi^2(2\lambda+c_0^2)^2} < t_0,
  \end{equation}
  i.e. $\min\{t_0,T\}=T$ and the claim follows.
\end{proof}

This already concludes the first part of the proof of \Cref{prop:neg-result}. The second part of the statement is proved in \Cref{prop:spherical-shrinker}, using the blow-up construction which is also fundamental in proving \Cref{thm:main}.

\section{Localized energy estimates}\label{sec:loc-en}

This section is devoted to the fundamental estimates which enable a suitable blow-up construction in later sections. To this end, following \cite[Section~3 respectively]{kuwertschaetzle2001,rupp2024,rupp2023}, one localizes the energy decay of $\int |A|^2\dd \mu$ to obtain a life-span theorem and suitable uniform bounds on derivatives of the curvature as long as an energy concentration is controlled. This contrasts the previous section which is concerned with results on surfaces which are energetically close to round spheres --- for these, already results on the decay of $\int|A^0|^2\dd\mu$ are important. For the life-span result however, we need to study possible concentrations of the full second fundamental form. Moreover, this section's computations and the life-span theorem are valid on general oriented, compact and closed surfaces $\Sigma$ of some genus $g\in\N_0$, not just for spheres.

While the computations here are very similar to past works, due to the scaling behavior in \Cref{lem:par-scal}, it is important to track the dependence of all estimates on the parameters $c_0$ and $\lambda$ in order to give a precise life-span result which applies for a blow-up sequence as constructed in the next section --- this is necessary since, along a blow-up, $c_0$ and $\lambda$ change as in \Cref{lem:par-scal}. While in \cite[Theorem~1.2]{liu2012}, a life-span theorem for the $(c_0,\lambda)$-Helfrich flow is already obtained, the dependence of the constants on $c_0$ and $\lambda$ is not explicitly tracked. Therefore, in the following computations, we give special attention to these parameters.

\begin{lemma}\label{lem:loc-1}
  Let $c_0,\lambda\in\R$ and $f\colon[0,T)\times\Sigma\to\R^3$ be a $(c_0,\lambda)$-Helfrich flow. If $\widetilde{\eta}\in C_c^{\infty}(\R^3)$ and $\eta=\widetilde{\eta}\circ f$, 
  \begin{align}
    \frac{\dd}{\dd t}&\int_{\Sigma}\frac12 H^2\eta\dd \mu + \int_{\Sigma} |\nabla\W_0(f)|^2\eta\dd\mu\\
    &= (\lambda+\frac12c_0^2) \int_{\Sigma}\Delta H H\eta\dd\mu + (\lambda+\frac12c_0^2)\int_{\Sigma}|A^0|^2H^2\eta\dd \mu\\
    &\quad-2\int_{\Sigma}\langle\nabla\W_0(f),\nu\rangle\langle\nabla H,\nabla\eta\rangle_{g_f}\dd\mu - \int_{\Sigma}\langle\nabla\W_0(f),\nu\rangle H\Delta\eta\dd\mu + \frac12\int_{\Sigma}H^2\partial_t\eta\dd\mu\\
    &\quad+c_0\int_{\Sigma}(|A^0|^2-\frac12H^2)(\langle\nabla\W_0(f),\nu\rangle\eta + 2\langle \nabla H,\nabla\eta\rangle_{g_f}+H\Delta\eta)\dd\mu
  \end{align}
  as well as  
  \begin{align}
    \frac{\dd}{\dd t}&\int_{\Sigma}|A^0|^2\eta\dd\mu+\int_{\Sigma}|\nabla\W_0(f)|^2\eta\dd\mu\\
    &= (2\lambda+c_0^2)\int_{\Sigma}\langle\nabla^2H,A^0\rangle_{g_f}\eta\dd\mu + (\lambda+\frac12c_0^2)\int_{\Sigma}|A^0|^2H^2\eta\dd\mu\\
    &\quad-2\int_{\Sigma}\langle\nabla\W_0(f),\nu\rangle\bigl(\langle A^0,\nabla^2\eta\rangle_{g_f}+\langle \nabla H,\nabla\eta\rangle_{g_f}\bigr)\dd\mu \\
    &\quad + c_0\int_{\Sigma}(|A^0|^2-\frac12H^2) \bigl( \langle\nabla\W_0(f),\nu\rangle\eta + 2\langle A^0,\nabla^2\eta\rangle_{g_f}+2\langle \nabla H,\nabla\eta\rangle_{g_f} \bigr)\dd\mu\\
    &\quad + \int_{\Sigma}|A^0|^2\partial_t\eta\dd\mu.
  \end{align}
\end{lemma}
\begin{proof}
  Using \eqref{eq:ev-mu}, \eqref{eq:ev-H} and \cite[(31) and (32)]{kuwertschaetzle2001}, writing $\partial_tf=\xi\nu$,
  \begin{align}
    \partial_t&\int_{\Sigma}\frac12H^2\eta\dd\mu + \int_{\Sigma}|\nabla\W_0(f)|^2\eta\dd\mu\\
    &= \int_{\Sigma}\lambda H(\Delta H+|A^0|^2H)\eta\dd\mu+\int_{\Sigma} (2\xi\langle\nabla H,\nabla\eta\rangle_{g_f}+H\xi\Delta\eta)\dd\mu\\
    &\quad+c_0\int_{\Sigma}(|A^0|^2-\frac12H(H-c_0))(\Delta H+|A^0|^2H)\eta\dd\mu+\frac12\int_{\Sigma}H^2\partial_t\eta\dd\mu.
  \end{align}
  Plugging in \eqref{eq:flow-eq} and using $\Delta(H\eta) = \Delta H\eta + 2\langle \nabla H,\nabla \eta\rangle_{g_f}+H\Delta \eta$, one finds
  \begin{align}
    \partial_t&\int_{\Sigma}\frac12H^2\eta\dd\mu + \int_{\Sigma}|\nabla\W_0(f)|^2\eta\dd\mu\\
    &= \int_{\Sigma}(\lambda+\frac12c_0^2) H(\Delta (H\eta)+|A^0|^2H\eta)\dd\mu\\
    &\quad-2\int_{\Sigma}\langle\nabla\W_0(f),\nu\rangle\langle\nabla H,\nabla\eta\rangle_{g_f}\dd\mu-\int_{\Sigma}\langle\nabla\W_0(f),\nu\rangle H\Delta\eta\dd\mu\\
    &\quad+c_0\int_{\Sigma}(|A^0|^2-\frac12H^2)(\langle\nabla\W_0(f),\nu\rangle\eta + 2\langle \nabla H,\nabla\eta\rangle_{g_f}+H\Delta\eta)\dd\mu+\frac12\int_{\Sigma}H^2\partial_t\eta\dd\mu.
  \end{align}
  The first claim follows after integrating by parts. Moreover, as in \cite[proof of Lemma B.1]{rupp2023} and \cite[p.~423]{kuwertschaetzle2001}, one finds in the coordinates of a local orthonormal frame
  \begin{align}
    \partial_t(|A^0|^2\dd\mu) &= 2\nabla_i(\nabla_j\xi A^0(e_i,e_j))\dd\mu-\nabla_j\xi\nabla_jH\dd\mu+|A^0|^2H\xi\dd\mu\\
    &=2\nabla_i(\nabla_j\xi A^0(e_i,e_j))\dd\mu-\nabla_j(\xi\nabla_jH)\dd\mu+(\Delta H+|A^0|^2H)\xi\dd\mu.
  \end{align}
  Again with \eqref{eq:flow-eq}, 
  \begin{align}
    \partial_t&(|A^0|^2\dd\mu) + |\nabla\W_0(f)|^2\dd\mu \\
    &= 2\nabla_i(\nabla_j\xi A^0(e_i,e_j))\dd\mu-\nabla_j(\xi\nabla_jH)\dd\mu\\
    &\quad + c_0(|A^0|^2-\frac12H^2)\langle\nabla\W_0(f),\nu\rangle \dd\mu+(\lambda+\frac12c_0^2) H \langle\nabla\W_0(f),\nu\rangle\dd\mu.
  \end{align}
  Integrating by parts and using $\nabla_iH=2(\nabla_jA^0)_{ij}$ by Codazzi-Mainardi, one obtains
  \begin{align}
    \partial_t&\int_{\Sigma}|A^0|^2\eta\dd\mu+\int_{\Sigma}|\nabla\W_0(f)|^2\eta\dd\mu\\
    &= \int_{\Sigma} 2\xi A^0_{ij}\nabla_{ij}^2\eta+2\xi\nabla_jH\nabla_j\eta + (\lambda+\frac12c_0^2) H\langle\nabla\W_0(f),\nu\rangle\eta\dd\mu\\
    &\quad +\int_{\Sigma}c_0(|A^0|^2-\frac12H^2)\langle\nabla\W_0(f),\nu\rangle \eta \dd\mu + \int_{\Sigma} |A^0|^2\partial_t\eta\dd\mu.
  \end{align}
  Integration by parts and $\nabla_iH=2(\nabla_jA^0)_{ij}$ yields
  \begin{equation}
    (\lambda+\frac12c_0^2)\int_{\Sigma} H(2A^0_{ij}\nabla^2_{ij}\eta+2\nabla_jH\nabla_j\eta+\Delta H\eta)\dd\mu = 2(\lambda+\frac12c_0^2)\int_{\Sigma}\langle\nabla^2H,A^0\rangle_{g_f}\eta\dd\mu,
  \end{equation}
  so that the claim follows.
\end{proof}

As in \cite{kuwertschaetzle2001,rupp2023,rupp2024}, consider $\widetilde{\gamma}\in C_c^{\infty}(\R^3)$ with $|\widetilde{\gamma}|\leq 1$ and define $\gamma=\widetilde{\gamma}\circ f$. Fix some $\Lambda > 0$ which satisfies $\|D\widetilde{\gamma}\|_{\infty}\leq\Lambda$ and $\|D^2\widetilde{\gamma}\|_{\infty}\leq\Lambda^2$. As on \cite[p. 8]{rupp2023}, one estimates
\begin{equation}\label{eq:gamma}
  |\nabla\gamma|\leq \Lambda\quad\text{and}\quad |\nabla^2\gamma|\leq C(\Lambda^2+|A|\Lambda).
\end{equation}

\begin{lemma}\label{lem:loc-2}
  In the setting of \Cref{lem:loc-1} and with $\gamma$ as above, one has
  \begin{align}
    \frac{\dd}{\dd t}&\int_{\Sigma} |A|^2\gamma^4\dd\mu + \frac32\int_{\Sigma}|\nabla\W_0(f)|^2\gamma^4\dd\mu\\
    &\leq C (\lambda+c_0^2)\int_{\Sigma}\bigl(|\nabla^2H||A|+|A|^4\bigr)\gamma^4+\Lambda|A|^3\gamma^3\dd\mu+ C c_0^2\int_{\Sigma}|A|^4\gamma^4\dd\mu\\
    &\quad + C\Lambda^4\int_{\{\gamma>0\}}|A|^2\dd\mu+C\Lambda^2\int_{\Sigma}|A|^4\gamma^2\dd\mu .
  \end{align}
\end{lemma}
\begin{proof}
  Using \eqref{eq:a-a0-h}, \Cref{lem:loc-1} and $\langle \nabla^2\varphi,A\rangle_{g_f}=\langle \nabla^2\varphi,A^0\rangle_{g_f}+\frac12H\Delta \varphi$ which applies for any $\varphi\in C^{\infty}([0,T)\times\Sigma)$, one finds
  \begin{align}
    &\frac{\dd}{\dd t}\int_{\Sigma}|A|^2\gamma^4\dd\mu+2\int_{\Sigma}|\nabla\W_0(f)|^2\gamma^4\dd\mu\\
    &= (2\lambda+c_0^2) \int_{\Sigma}\bigl(\langle\nabla^2H,A\rangle_{g_f}+|A^0|^2H^2\bigr)\gamma^4\dd\mu\\
    &\quad -2\int_{\Sigma}\langle\nabla\W_0(f),\nu\rangle\bigl(2\langle\nabla H,\nabla\gamma^4\rangle_{g_f}+\langle\nabla^2\gamma^4,A\rangle_{g_f}\bigr)\dd\mu+\int_{\Sigma}|A|^2\partial_t\gamma^4\dd\mu\\
    &\quad + c_0\int_{\Sigma}(|A^0|^2-\frac12H^2)\bigl( 2\langle\nabla^2\gamma^4,A\rangle_{g_f}+4\langle\nabla H,\nabla\gamma^4\rangle_{g_f}+2\langle\nabla\W_0(f),\nu\rangle\gamma^4 \bigr)\dd\mu.
  \end{align}
  We now estimate the individual terms. A straight-forward computations yields 
  \begin{equation}
    |\partial_t\gamma^4|\leq C\Lambda\gamma^3\bigl( |\nabla\W_0(f)|+(|\lambda|+c_0^2)|A|+|c_0||A|^2 \bigr).
  \end{equation}
  Moreover, using $|\nabla^2\gamma^4|\leq C(\Lambda^2+\Lambda |A|\gamma)\gamma^2$, for some $\eta>0$,
  \begin{align}
    \int_{\Sigma} |\nabla\W_0(f)||\nabla^2\gamma^4||A|\dd\mu &\leq \eta \int_{\Sigma}|\nabla\W_0(f)|^2\gamma^4\dd\mu + C(\eta)\Lambda^2\int_{\Sigma}|A|^4\gamma^2\dd\mu\\
    &\quad+C(\eta)\Lambda^4\int_{\{\gamma>0\}}|A|^2\dd\mu.
  \end{align}
  Furthermore, as in \cite[proof of Lemma 3.2]{kuwertschaetzle2001}, one obtains
  \begin{equation}\label{eq:loc-2-1}
    \int_{\Sigma}|\nabla H|^2\gamma^2\dd\mu\leq \eta \int_{\Sigma} |\nabla\W_0(f)|^2\gamma^4\dd\mu + \frac{C}{\eta}\Lambda^2\int_{\{\gamma>0\}} H^2\dd\mu + C\int_{\Sigma}|A|^4\gamma^2\dd\mu.
  \end{equation}
  Therefore, 
  \begin{align}
     \int_{\Sigma}|\nabla\W_0(f)||\nabla H||\nabla \gamma^4|\dd\mu &\leq \eta \int_{\Sigma}|\nabla\W_0(f)|^2\gamma^4\dd\mu+C(\eta)\Lambda^2\int_{\Sigma}|A|^4\gamma^2\dd\mu\\
     &\quad+C(\eta)\Lambda^4\int_{\{\gamma>0\}}H^2\dd\mu.
  \end{align}
  The terms including $c_0$ in the last line can be dealt with using Young's inequality and again \eqref{eq:loc-2-1}, $|\nabla^2\gamma^4|\leq C(\Lambda^2+|A|\Lambda\gamma)\gamma^2$ and, by \eqref{eq:cod-main} 
  \begin{equation}
    \bigl|c_0(|A^0|^2-\frac12H^2)\bigr|\leq C |c_0||A|^2,
  \end{equation}
  proceeding similarly as above.
\end{proof}

\begin{proposition}\label{prop:en-est}
  There exist universal constants $\varepsilon_0,\delta_0,C\in(0,\infty)$ such that, for any $f$ and $\gamma$ as in \Cref{lem:loc-2}, if
  \begin{equation}
    \int_{\{\gamma>0\}} |A|^2\dd\mu<\varepsilon_0\quad\text{at some time $t\in[0,T)$},
  \end{equation}
  then at time $t$ one has
  \begin{align}
    \frac{\dd}{\dd t}&\int_{\Sigma}|A|^2\gamma^4\dd\mu + \delta_0 \int_{\Sigma} \bigl( |\nabla^2 A|^2+|A|^2|\nabla A|^2+|A|^6 \bigr)\gamma^4\dd\mu\\
    &\leq C\Lambda^4\int_{\{\gamma>0\}}|A|^2\dd\mu + C(\lambda^2+c_0^4)\int_{\Sigma}|A|^2\gamma^4\dd\mu.
  \end{align}
\end{proposition}
\begin{proof}
  By \cite[Proposition 3.2]{rupp2023}, \cite[Proposition 2.6 and Lemma 4.2]{kuwertschaetzle2001}, at time $t$, the following interpolation inequality holds.
  \begin{equation}
    \int_{\Sigma} \bigl( |\nabla^2 A|^2+|A|^2|\nabla A|^2+|A|^6 \bigr)\gamma^4\dd\mu\leq C\int_{\Sigma}|\nabla\W_0(f)|^2\gamma^4\dd\mu+C\Lambda^4\int_{\{\gamma>0\}}|A|^2\dd\mu.
  \end{equation}
  Using \Cref{lem:loc-2}, there thus exists $\delta_0\in(0,\infty)$ with
  \begin{align}
    \frac{\dd}{\dd t}&\int_{\Sigma} |A|^2\gamma^4\dd\mu + 2\delta_0\int_{\Sigma} \bigl( |\nabla^2 A|^2+|A|^2|\nabla A|^2+|A|^6 \bigr)\gamma^4\dd\mu\\
    &\leq C(\lambda+c_0^2)\int_{\Sigma}\bigl(|\nabla^2H||A|+|A|^4\bigr)\gamma^4+\Lambda|A|^3\gamma^3\dd\mu+ C c_0^2 \int_{\Sigma}|A|^4\gamma^4\dd\mu\\
    &\quad + C\Lambda^4\int_{\{\gamma>0\}}|A|^2\dd\mu+C\Lambda^2\int_{\Sigma}|A|^4\gamma^2\dd\mu .
  \end{align}
  Proceeding as in \cite[Proposition 3.3]{rupp2024}, one finds that
  \begin{align}
    \frac{\dd}{\dd t}&\int_{\Sigma} |A|^2\gamma^4\dd\mu + \frac32\delta_0\int_{\Sigma} \bigl( |\nabla^2 A|^2+|A|^2|\nabla A|^2+|A|^6 \bigr)\gamma^4\dd\mu\\
    &\leq C\Lambda^4\int_{\{\gamma>0\}}|A|^2\dd\mu + C(\lambda^2+c_0^4)\int_{\Sigma}|A|^2\gamma^4\dd\mu + C c_0^2\int_{\Sigma}|A|^4\gamma^4\dd\mu.
  \end{align}
  The claim follows noting that $c_0^2\int_{\Sigma}|A|^4\gamma^4\dd\mu\leq \eta \int_{\Sigma}|A|^6\gamma^4\dd\mu + C(\eta)c_0^4 \int_{\Sigma}|A|^2\gamma^4\dd\mu$.
\end{proof}

Still following the arguments of \cite{kuwertschaetzle2001,rupp2023,rupp2024}, we introduce the following notation.

\begin{definition}
  Consider a family of immersions $f\colon[0,T)\times\Sigma\to\R^3$. For $t\in[0,T)$ and $r>0$, the \emph{concentration of curvature function} is defined as 
  \begin{equation}
    \scurv(t,r)=\sup_{x\in\R^3} \int_{B_r(x)} |A|^2\dd\mu.
  \end{equation} 
\end{definition}

With this notation, one obtains the following corollary as an integrated version of \Cref{prop:en-est}.

\begin{corollary}\label{cor:en-est}
  Let $\varepsilon_0\in(0,\infty)$, $\delta_0\in(0,\infty)$ as in \Cref{prop:en-est}. There exists a uniform constant $C>0$ such that, for $c_0,\lambda\in\R$ and a $(c_0,\lambda)$-Helfrich flow $f\colon[0,T)\times\Sigma\to\R^3$, one has the following.

  If there exists $\rho>0$ such that
  \begin{equation}
    \scurv(t,\rho)<\varepsilon_0\quad\text{for all $t\in[0,T)$},
  \end{equation}
  then, for all $x\in\R^3$ and $0\leq t<T$,
  \begin{align}
    &\int_{B_{\frac{\rho}{2}}(x)}|A|^2\dd\mu\Big|_t + \delta_0\int_{0}^{t}\int_{B_{\frac{\rho}{2}}(x)}\bigl( |\nabla^2A|^2+|A|^2|\nabla A|^2+|A|^6 \bigr)\dd\mu\dd\tau\\
    &\leq \int_{B_{\rho}(x)}|A|^2\dd\mu\Big|_{t=0} + \frac{C}{\rho^4} \int_{0}^{t}\int_{B_{\rho}(x)}|A|^2\dd\mu\dd\tau + C(\lambda^2+c_0^4)\int_{0}^{t}\int_{B_{\rho}(x)}|A|^2\dd\mu\dd\tau.
  \end{align}
\end{corollary}
\begin{proof}
  Choose $\widetilde{\gamma}\in C_c^{\infty}(\R^3)$ with $\chi_{B_{\frac{\rho}{2}}(x)}\leq \widetilde{\gamma}\leq \chi_{B_{\rho}(x)}$, $\|D\widetilde{\gamma}\|_{\infty}\leq C/\rho$ and $\|D^2\widetilde{\gamma}\|_{\infty}\leq C/\rho^2$. The claim follows by simply integrating the estimate in \Cref{prop:en-est} in time with $\Lambda=C/\rho$.
\end{proof}

\section{Construction of a blow-up}\label{sec:blow-up}

\subsection{Bounds on higher-order derivatives of the second fundamental form}

This section's main result is the following analog of \cite[Proposition~3.7]{rupp2024} which in turn is based on \cite[Theorem~3.5]{kuwertschaetzle2001}, using the higher-order interpolation estimates in \Cref{prop:ho-est}.

\begin{proposition}\label{prop:ho-en-est}
  Let $c_0,\lambda\in\R$ and $\varepsilon_0$ as in \Cref{prop:en-est}. Further consider a $(c_0,\lambda)$-Helfrich flow $f\colon[0,T)\times\Sigma\to\R^3$ and suppose that $\rho>0$ satisfies $T\leq T^*\rho^4$ for some $0<T^*<\infty$ and
  \begin{equation}
    \scurv(t,\rho)\leq\varepsilon<\varepsilon_0\quad\text{for all $0\leq t<T$}.
  \end{equation}
  Further, assume 
  \begin{equation}\label{eq:dep-c0-lam}
    (\lambda^2+c_0^4)T\leq\overline{L}<\infty.
  \end{equation}
  For all $t\in(0,T)$ and $m\in\N_0$, one has the local estimates
  \begin{align}
    \|\nabla^m A\|_{L^2(B_{\frac{\rho}{8}}(x))}&\leq C(m,T^*,\overline{L})\sqrt{\varepsilon}t^{-\frac{m}{4}},\\
    \|\nabla^m A\|_{L^{\infty}}&\leq C(m,T^*,\overline{L})\sqrt{\varepsilon}t^{-\frac{m+1}{4}},\label{eq:ho-en-est-loc}
  \end{align}
  and the full $L^2(\dd\mu)$-bounds
  \begin{equation}
    \|\nabla^m A\|_{L^2(\dd\mu)}\leq C(m,T^*,\overline{L})t^{-\frac{m}{4}}\Bigl(\int_{\Sigma}|A|^2\dd\mu\big|_{t=0}\Bigr)^{\frac12}.\label{eq:ho-en-est-full}
  \end{equation}
\end{proposition}
\begin{proof}[Sketch of a proof.]
  Especially using that \eqref{eq:dep-c0-lam} is invariant with respect to parabolic rescaling by \Cref{lem:par-scal}, we may w.l.o.g. suppose $\rho=1$. Then setting $K(t)=\int_{B_1(0)}|A|^2\dd\mu$ and $L(t)\equiv (\lambda^2+c_0^4)$, one can adapt the exact same proof as for \cite[Proposition~3.7]{rupp2024}, using \Cref{cor:en-est} and \Cref{prop:ho-est}. So the details can be safely omitted here.
\end{proof}
\begin{remark}\label{rem:dep-c0-lam}
  The main feature of \Cref{prop:ho-en-est} is how the constants in the above estimates depend on $c_0$ and $\lambda$. That is an important insight since we later apply \Cref{prop:ho-en-est} to a blow-up sequence of a $(c_0,\lambda)$-Helfrich flow where $c_0$ and $\lambda$ change. \Cref{lem:par-scal} and \eqref{eq:dep-c0-lam} show that the scaling behavior of $c_0$, $\lambda$ and the time $T$ under parabolic scaling of the Helfrich flow cancels exactly in the right way in the constants of \Cref{prop:ho-en-est} --- as it should be.
\end{remark}

\subsection{Life-span result}

\begin{proposition}\label{prop:life}
  Let $\varepsilon_0>0$ be as in \Cref{prop:en-est}. Then there exists a universal constant $0<\overline{\varepsilon}<\min\{\varepsilon_0,8\pi\}$ with the following property.

  If $c_0,\lambda\in\R$ and $f$ is a $(c_0,\lambda)$-Helfrich flow with
  \begin{equation}
    \scurv(0,\rho)\leq\varepsilon<\overline{\varepsilon}\quad\text{for some $\rho>0$},
  \end{equation}
  then the maximal existence time $T$ of $f$ satisfies $\bigl(1+\rho^4(\lambda^2+c_0^4)\bigr)T>\hat{c}\rho^4$ for some universal $\hat{c}\in(0,1)$ and
  \begin{equation}
    \scurv(t,\rho)\leq\hat{c}^{-1}\varepsilon\quad\text{for all $t\in\bigl[0,\frac{\hat{c}\rho^4}{1+\rho^4(\lambda^2+c_0^4)}\bigr]$}.
  \end{equation}
\end{proposition}
\begin{proof}
  After scaling as in \Cref{lem:par-scal}, we may suppose w.l.o.g. that $\rho=1$. Denoting by $\Gamma>1$ the minimal number of balls of radius $\frac12$ necessary to cover $B_1(0)\subseteq\R^3$, one clearly has
  \begin{equation}\label{eq:life-prop-1}
    \scurv(t,1)\leq \Gamma \cdot\scurv(t,\frac{1}{2}).
  \end{equation}
  Then choose $\overline{\varepsilon}=\frac{\varepsilon_0}{3\Gamma}$. Note that $\scurv(t)=\scurv(t,1)$ is a continuous function with $\scurv(0)\leq\varepsilon<\overline{\varepsilon}$. Therefore, for the universal constant $\omega=(6\widetilde{C}\Gamma)^{-1}$ where $\widetilde{C}$ is the universal constant ``$\,C\,$'' from \Cref{cor:en-est},
  \begin{equation}
    t_0\vcentcolon= \sup\{t\in[0,\min\{T,\frac{\omega}{1+c_0^4+\lambda^2}\}):\scurv\leq3\Gamma\varepsilon\text{ on $[0,t)$}\} > 0.
  \end{equation}
  By the choice of $\overline{\varepsilon}$, we have $\scurv(t)\leq 3\Gamma\varepsilon < \varepsilon_0$ for $t\in[0,t_0)$. Thus, by \Cref{cor:en-est},
  \begin{align}
    \int_{B_{\frac12}(x)}|A|^2\dd\mu\leq \int_{B_1(x)}|A|^2\dd\mu\Big|_{t=0} + 3 \widetilde{C} (1+c_0^4+\lambda^2)\Gamma\varepsilon t,
  \end{align} 
  for all $0\leq t<t_0$. Then
  \begin{equation}
    \int_{B_{\frac12}(x)}|A|^2\dd\mu\leq \int_{B_1(x)}|A|^2\dd\mu\Big|_{t=0} + \frac{\varepsilon}{2}\frac{1+c_0^4+\lambda^2}{\omega}t \leq 2\varepsilon
  \end{equation}
  for all $0\leq t<t_0$. Using \eqref{eq:life-prop-1}, one finds $\scurv(t)\leq 2\Gamma\varepsilon$ for all $0\leq t<t_0$. By the definition of $t_0$ and continuity of $\scurv$, if $t_0<\min\{T,\frac{\omega}{1+c_0^4+\lambda^2}\}$, one also has $\scurv(t_0)=3\Gamma\varepsilon$, a contradiction!

  \textbf{Case 1: } $t_0=T<\omega/(1+c_0^4+\lambda^2)$. Using the specific choices for $t_0$ and $\overline{\varepsilon}$ made above, one again obtains $\scurv(t)\leq 3\Gamma\varepsilon<\varepsilon_0$ for $0\leq t<T$. So \Cref{prop:ho-en-est} yields for $0<\zeta<T$
  \begin{align}
    \|\nabla^m A\|_{L^{\infty}}  &\leq C(m,T,(\lambda^2+c_0^4)T,\zeta)\quad\text{and}\\
    \|\nabla^m A\|_{L^{2}(\dd\mu)}  &\leq C(m,T,(\lambda^2+c_0^4)T,\zeta,\W(f_0),g)
  \end{align}
  for $t\in[\zeta,T)$, respectively, especially using \eqref{eq:int-asq} in the second estimate. With the same arguments as in \cite[pp. 330 -- 332]{kuwertschaetzle2002}, one deduces that $f(t)$ converges smoothly to an immersion $f(T)$ as $t\nearrow T$. Thus, one can restart the flow at time $T$ which contradicts the maximality of $T$, using \Cref{rem:well-posedness}. Therefore, the only possible case left is
  
  \textbf{Case 2: }$t_0=\omega/(1+c_0^4+\lambda^2)\leq T$. Choosing $\hat{c}=\min\{\omega,(2\Gamma)^{-1},1\}>0$, the claim is proved.
\end{proof}

\begin{remark}
  In \cite[Theorem~3.1]{mccoywheeler2016}, a life-span result for flows of a general structure which also applies for \eqref{eq:flow-eq} is proved. However, as used in the proof of their result on \cite[p. 864]{mccoywheeler2016}, the flows to which \cite[Theorem~3.1]{mccoywheeler2016} applies are invariant with respect to parabolic scaling. As this does not apply to \eqref{eq:flow-eq}, cf. \Cref{lem:par-scal}, we provide all arguments for \Cref{prop:life} above in order to precisely understand how the parameters $c_0$ and $\lambda$ come into play.
\end{remark}

\subsection{Properties of a blow-up sequence and concentration limit}

First, consider $c_0,\lambda\in\R$ and a maximal $(c_0,\lambda)$-Helfrich flow with initial datum $f_0$. For sequences of times $(t_j)_{j\in\N}$ with $t_j\nearrow T$, radii $(r_j)_{j\in\N}\subseteq(0,\infty)$ and points $(x_j)_{j\in\N}\subseteq\R^3$, define the rescaled flows
\begin{equation}
  f_j\colon\big[0,\frac{1}{r_j^4}(T-t_j)\big)\times\Sigma\to\R^3,\quad f_j(t,p)=\frac{f(t_j+r_j^4t,p)-x_j}{r_j}\,.
\end{equation}

\begin{lemma}\label{lem:ex-blow-up}
  Suppose that $\W(f(t))\leq M$ and $\A(f(t))\leq A$ for all $0\leq t<T$. With $\hat{c}$ and $\overline{\varepsilon}$ from \Cref{prop:life}, one can choose the sequences $(t_j)_{j\in\N}$, $(r_j)_{j\in\N}$ and $(x_j)_{j\in\N}$ such that the following properties hold. For the radii, there exists $r_{\max}=r_{\max}(M,A)$ with $0<r_j\leq r_{\max}$ for all $j\in\N$. Furthermore,
  \begin{enumerate}[(i)]
    \item $t_j+\hat{c}\frac{r_j^4}{1+r_j^4(\lambda^2+c_0^4)} < T$;
    \item $\scurv_j(t,1)\leq \overline{\varepsilon}$ for all $0\leq t\leq \hat{c}\frac{1}{1+r_j^4(\lambda^2+c_0^4)}$;
    \item $\inf_{j\in\N} \int_{B_1(0)} |A_{f_j(\tilde{c})}|^2\dd\mu_{f_j(\tilde{c})}>0$ where $0<\tilde{c}\vcentcolon=\hat{c}\frac{1}{1+r_{\max}^4(\lambda^2+c_0^4)} \leq \hat{c}\frac{1}{1+r_{j}^4(\lambda^2+c_0^4)}$.
  \end{enumerate}
\end{lemma}
\begin{proof}
  With the arguments in \cite[Lemma~6.6 in Article~C]{rupp2022PhD}, there are $\alpha>0$ and $(r_t)_{t\in[0,T)}\subseteq(0,\infty)$ with
  \begin{equation}\label{eq:ex-bu-1}
    \alpha < \scurv(t,r_t) < \hat{c} \overline{\varepsilon} \quad\text{for all $0\leq t<T$}.
  \end{equation}
  \begin{claim}\label{cl:bu-1}
    One has 
    \begin{equation}\label{eq:cl-bu-1}
      r_t \leq\frac12\mathrm{diam}(f(t,\Sigma))\leq \frac{1}{\pi} \sqrt{\W(f(t))}\sqrt{\A(f(t))} \leq r_{\max}(M,A).
    \end{equation}
  \end{claim}
  Note that, by Simons's diameter estimate, cf. \cite[Lemma~1.1]{topping1998}, one only needs to show the first inequality in the above. To this end, if $r_t>\frac12\mathrm{diam}(f(t,\Sigma))$, there is $0<r<r_t$ and a ball $B_r(x)\subseteq\R^3$ with $f(t,\Sigma)\subseteq B_r(x)$. Using $\|A\|^2_{L^2(\dd\mu)}\geq 8\pi$, one finds
  \begin{equation}
    8\pi \leq \int_{B_r(x)}|A|^2\dd\mu\Big|_{t} \leq \scurv(t,r_t) \leq \hat{c}\overline{\varepsilon},
  \end{equation}
  a contradiction to the choice of $\overline{\varepsilon}$. So \Cref{cl:bu-1} is established.
  
  We now argue 
  \begin{equation}
    \liminf_{t\nearrow T}\frac{r_{t+r_{t}^4\tilde{c}}}{r_t} < 2.
  \end{equation}
  Otherwise, there exists $t_0\in (0,T)$ with $r_{t+r_t^4\tilde{c}}\geq 2r_t$ for all $t_0\leq t<T$. Define $s_1=t_0$ and, for $j\in\N$, $s_{j+1}=s_{j}+r_{s_j}^4\tilde{c}$. Then $r_{s_{j+1}}\geq 2^j r_{s_1}$ which contradicts \Cref{cl:bu-1}. Therefore, we can choose a sequence $t_j\nearrow T$ such that, for some $0<M<2$, writing $r_j=r_{t_j}$,
  \begin{equation}
    r_{t_j+r_j^4\tilde{c}} \leq M r_j\quad\text{for all $j\in\N$}.
  \end{equation}
  By a simple covering argument, there exists $N=N(M)>0$ such that
  \begin{equation}
    \scurv(t,M r) \leq N \scurv(t,r)\quad\text{for all $t\in[0,T)$ and $r>0$}.
  \end{equation}
  Therefore, one finds
  \begin{align}
    \scurv(t_j+r_j^4\tilde{c},r_j) \geq \frac{1}{N} \scurv(t_j+r_j^4\tilde{c},Mr_j) \geq \frac{1}{N} \scurv(t_j+r_j^4\tilde{c},r_{t_j+r_j^4\tilde{c}}) \geq \frac{\alpha}{N}.
  \end{align}
  Whence, choosing $x_j\in\R^3$ with 
  \begin{equation}
    \int_{B_{r_j}(x_j)} |A|^2\dd\mu \Big|_{t=t_j+r_j^4\tilde{c}}\geq \frac{\alpha}{2N},
  \end{equation}
  \eqref{eq:ex-bu-1} and \Cref{prop:life} yield (i), (ii) and (iii).
\end{proof}

The following proposition shows the existence and contains relevant properties of a so-called ``concentration limit'' obtained by this blow-up procedure.

\begin{proposition}\label{prop:conc-limit}
  Let $c_0>0$ and $\lambda>0$ and consider a maximal $(c_0,\lambda)$-Helfrich flow. Then the construction of \Cref{lem:ex-blow-up} applies and there exists a complete, orientable surface $\hat{\Sigma}\neq\emptyset$ without boundary and a proper immersion $\hat{f}\colon\hat{\Sigma}\to\R^3$ such that, after passing to a subsequence, $r_j\to r\in[0,\infty)$ and,
  \begin{enumerate}[(i)]
    \item as $j\to\infty$, $\hat{f}_j\vcentcolon= f_j(\tilde{c})\to\hat{f}$ smoothly on compact subsets of $\R^3$, after reparametrization;
    \item $\int_{B_1(0)}|\hat{A}|^2\dd\hat{\mu}>0$ and $\H_{rc_0,r^2\lambda}(\hat{f})\leq\lim_{j\to\infty} \H_{r_jc_0,r_j^2\lambda}(\hat{f}_j)=\lim_{t\nearrow T}\H_{c_0,\lambda}(f(t))$;
    \item $\hat{f}$ is an $(rc_0,r^2\lambda$)-Helfrich immersion and
    \item if $\A(\hat{f}_j)\to\infty$, then $r=0$ and $\hat{f}$ is a Willmore immersion. 
  \end{enumerate}
\end{proposition}
\begin{proof}
  Write $K=\H_{c_0}(f_0)+\frac12\lambda\A(f_0)$. Using \eqref{eq:en-dec}, one finds $\A(f(t))\leq 2K/\lambda$ for all $t\in[0,T)$. Moreover, by \eqref{eq:est-Will-energy} and \eqref{eq:en-dec}, also
  \begin{equation}\label{eq:conc-limit-0}
    \W(f(t))\leq \frac{2\lambda+c_0^2}{2\lambda} K
  \end{equation}
  is uniformly bounded in $t\in[0,T)$. Particularly, \Cref{lem:ex-blow-up} applies. After passing to a subsequence without relabeling, one may suppose $r_j\to r\in[0,\infty)$. By \Cref{lem:par-scal}, each $f_j$ is an $(r_jc_0,r_j^2\lambda)$-Helfrich flow which, by \Cref{lem:ex-blow-up}, exists on $[0,\widetilde{c}]\times\Sigma$, uniformly in $j\in\N$. Moreover, 
  \begin{equation}
    r_j^4(c_0^4+\lambda^2) \cdot \widetilde{c} \leq r_j^4(c_0^4+\lambda^2) \frac{\hat{c}}{1+r_j^4(c_0^4+\lambda^2)} \leq\hat{c}.
  \end{equation}
  Thus, by \Cref{prop:ho-en-est}, using \eqref{eq:conc-limit-0} and \eqref{eq:int-asq}, we have for $0<t\leq\widetilde{c}$
  \begin{align}
    \|\nabla^mA_j\|_{\infty}&\leq C(m,\hat{c})t^{-\frac{m+1}{4}},\\
    \|\nabla^mA_j\|_{L^2(\dd\mu_j)} &\leq C(m,\hat{c},\frac{2\lambda+c_0^2}{2\lambda}K,g)t^{-\frac{m}{4}}.\label{eq:conc-limit-1}
  \end{align}
  Further, from Simon's monotonicity formula (cf. \cite[Lemma~4.1]{kuwertschaetzle2001}) and \eqref{eq:conc-limit-0}, for any $R>0$, we find
  \begin{equation}
    R^{-2}\mu_j(B_R(0))\leq CK<\infty\quad\text{for all $j\in\N$}.
  \end{equation} 
  Altogether, by the localized version of Langer's compactness theorem (cf. \cite[Theorem 4.2]{kuwertschaetzle2001} and \cite[Appendix~A]{rupp2023}) applied to the sequence $\hat{f_j}=f_j(\widetilde{c})$, after passing to a subsequence, we have (i) with $\hat{f}$ as claimed. That is, there exists a complete surface $\hat{\Sigma}$ without boundary, a proper immersion $\hat{f}\colon\hat{\Sigma}\to\R^3$, diffeomorphisms $\phi_j\colon\hat{\Sigma}(j)\to U_j$ where $U_j\subseteq\Sigma_j$ are open sets and $\hat{\Sigma}(j)=\{p\in\hat{\Sigma}:|\hat{f}(p)|<j\}$, and $u_j\in C^{\infty}(\hat{\Sigma}(j),\R^3)$ such that
  \begin{equation}
    \hat{f}_j\circ\phi_j=\hat{f}+u_j\quad\text{in $\hat{\Sigma}(j)$}
  \end{equation}
  and $\|\hat{\nabla}^mu_j\|_{L^{\infty}(\hat{\Sigma}(j))}\to 0$ as $j\to\infty$, for all $m\in\N_0$. Now the smooth convergence on compact sets in $\hat{\Sigma}$ and (iii) in \Cref{lem:ex-blow-up} yield the first statement in (ii) and particularly that $\hat{\Sigma}\neq\emptyset$.

  Then consider the flows $\widetilde{f}_j\colon(0,\widetilde{c}]\times\hat{\Sigma}(j)\to\R^3$, $\widetilde{f}_j=f_j\circ\phi_j$. Note that these flows also satisfy the $L^{\infty}$ estimates in \eqref{eq:conc-limit-1}. Moreover,
  \begin{equation}\label{eq:conc-limit-2}
    \partial_t\widetilde{f}_j = -\bigl[\Delta\widetilde{H}_j+|\widetilde{A}^0_j|^2\widetilde{H}_j-r_jc_0\bigl(|\widetilde{A}^0_j|^2-\frac12 \widetilde{H}_j^2\bigr)-r_j^2(\lambda+\frac12c_0^2)\widetilde{H}_j\bigr](\nu_j\circ\phi_j).
  \end{equation}
  Fix $0<\zeta<\widetilde{c}$. As in \cite[pp. 331 -- 332]{kuwertschaetzle2002}, the $L^{\infty}$-estimates in \eqref{eq:conc-limit-1} combined with the evolution equation \eqref{eq:conc-limit-2}, the convergence of $\widetilde{f}_j(\tilde{c})=\hat{f}+u_j$ for $j\to\infty$ and $r<\infty$ yield in any local chart $(U,\psi)$ of $\hat{\Sigma}$
  \begin{equation}
    \|\partial^m \widetilde{f}_j\|_{L^{\infty}([\zeta,\widetilde{c}]\times U)},\ \|\partial^m\partial_t\widetilde{f}_j\|_{L^{\infty}([\zeta,\widetilde{c}]\times U)} \leq C(\zeta,m,\hat{c},r)
  \end{equation}
  for all $m\in\N_0$ where $\partial^m$ denotes the coordinate derivative in the local chart $(U,\psi)$. Arguing as in \cite[Article~C, pp. 25 -- 26]{rupp2022PhD}, one concludes that, after passing to a further subsequence, there exists a family of immersions $\widetilde{f}\colon[\zeta,\widetilde{c}]\times\hat{\Sigma}\to\R^3$ such that, for any compact $P\subseteq\hat{\Sigma}$, $\widetilde{f}_j\to\widetilde{f}$ in $C^1([\zeta,\widetilde{c}],C^m(P))$, for any $m\in\N$. Particularly, $\nu_j\circ\phi_j\to\widetilde{\nu}$, a smooth, globally defined normal vector field on $\hat{\Sigma}$ and thus, $\hat{\Sigma}$ is orientable. Moreover, 
  \begin{equation}
    \partial_t\widetilde{f} =  -\bigl[\Delta\widetilde{H}+|\widetilde{A}^0|^2\widetilde{H}-rc_0\bigl(|\widetilde{A}^0|^2-\frac12 \widetilde{H}^2\bigr)-r^2(\lambda+\frac12c_0^2) \widetilde{H}\bigr]\widetilde{\nu}\quad\text{on $[\zeta,\widetilde{c}]\times\hat{\Sigma}$}.
  \end{equation}
  Furthermore, for $P\subseteq\hat{\Sigma}$ compact, using that due to \Cref{rem:en-dec}, the limit $\lim_{t\nearrow T}\H_{c_0,\lambda}(f(t))$ exists due to monotonicity,
  \begin{align}
    \int_{\zeta}^{\tilde{c}} \int_P |\partial_t\widetilde{f}|^2\dd\widetilde{\mu}\dd t &= \lim_{j\to\infty} \int_{\zeta}^{\tilde{c}} \int_P |\partial_t\widetilde{f}_j|^2\dd\widetilde{\mu}_j\dd t \leq \liminf_{j\to\infty} \int_{\zeta}^{\tilde{c}} \int_{\Sigma} |\partial_tf_j|^2\dd\mu_j\dd t\\
    &=2\liminf_{j\to\infty} \bigl( \H_{c_0,\lambda}(f(t_j+r_j^4\zeta))-\H_{c_0,\lambda}(f(t_j+r_j^4\tilde{c}))\bigr) = 0.
  \end{align}
  So $\hat{f}=\widetilde{f}(\tilde{c})$ is an $(rc_0,r^2\lambda)$-Helfrich immersion. Moreover, using \Cref{rem:en-dec}, for $P\subseteq\hat{\Sigma}$ compact,
  \begin{align}
    \frac14\int_{P}(\widetilde{H}-rc_0)^2&\dd\widetilde{\mu} + r^2\frac12\lambda\widetilde{\mu}(P) \Big|_{t=\tilde{c}} = \lim_{j\to\infty} \frac14 \int_{P}(\widetilde{H}_j-r_jc_0)^2\dd\widetilde{\mu}_j + r_j^2\frac12\lambda\widetilde{\mu}_j(P)\Big|_{t=\tilde{c}}\\
    &\leq\liminf_{j\to\infty}\H_{r_jc_0,r_j^2\lambda}(\widetilde{f}_j(\tilde{c})) \leq \liminf_{j\to\infty} \H_{r_jc_0,r_j^2\lambda}(f_j(\tilde{c})) \\
    &= \liminf_{j\to\infty} \H_{c_0,\lambda}(f(t_j+r_j^4\tilde{c})) = \lim_{t\nearrow T} \H_{c_0,\lambda}(f(t)).
  \end{align}
  Since this holds for any $P\subseteq\hat{\Sigma}$ compact, one finds $\H_{rc_0,r^2\lambda}(\hat{f}) \leq \lim_{j\to\infty}\H_{r_jc_0,r_j^2\lambda}(\hat{f}_j)=\lim_{t\nearrow T} \H_{c_0,\lambda}(f(t))$.

  By part (iii), for statement (iv), one only needs to check that $\A(\hat{f}_j)\to\infty$ implies $r=0$. This however is an immediate consequence of 
  \begin{equation}
    \A(\hat{f}_j) = \frac{1}{r_j^2} \A(f(t_j+\hat{t}_jr_j^4)) \leq \frac{2K}{\lambda}\frac{1}{r_j^2},
  \end{equation}  
  using the global area-bound for the flow established at the beginning of the proof.
\end{proof}

Throughout this article, the following lemma is used repeatedly. It summarizes the arguments originally employed in the convergence proof of the Willmore flow in \cite[Section~5]{kuwertschaetzle2004} and excludes non-compact concentration limits for Helfrich flows which are topologically spheres if the Willmore energy remains uniformly below $8\pi$.

\begin{lemma}\label{lem:kusch-rem-pt-sing}
  Consider a sequence of spherical immersions $\hat{f}_j\colon\S^2\to\R^3$ and a proper Willmore immersion $\hat{f}\colon\hat{\Sigma}\to\R^3$ of a complete, orientable surface $\hat{\Sigma}\neq\emptyset$ without boundary satisfying
  \begin{equation}\label{eq:no-plane-condition}
    \int_{B_1(0)}|\hat{A}|^2\dd\hat{\mu} >0.
  \end{equation}
  Suppose that $\hat{f}_j\to\hat{f}$ on compact subsets of $\R^3$, i.e. for $\hat{\Sigma}(j)=\{p\in\hat{\Sigma}:|\hat{f}(p)|<j\}$ there exist $u_j\in C^{\infty}(\hat{\Sigma}(j),\R^3)$ and diffeomorphisms $\phi_j\colon\hat{\Sigma}(j)\to\phi_j(\hat{\Sigma}(j))\subseteq\S^2$ such that 
  \begin{equation}\label{eq:kusch-rem-pt-sing-1}
    \hat{f}_j \circ\phi_j=\hat{f}+u_j\quad\text{in $\hat{\Sigma}(j)$}
  \end{equation}
  and $\|\hat{\nabla}^mu_j\|_{L^{\infty}(\hat{\Sigma}(j))}\to 0$ as $j\to\infty$, for all $m\in\N_0$.

  If there exists $\beta>0$ with $\W(\hat{f}_j)\leq 8\pi-\beta$ for all $j\in\N$, then $\hat{\Sigma}$ is diffeomorphic to $\S^2$. Writing $\hat{\Sigma}=\S^2$, there exist diffeomorphisms $\Phi_j\colon\S^2\to\S^2$ such that $\hat{f}_j\circ\Phi_j\to\hat{f}$ smoothly on $\S^2$ and $\hat{f}$ parametrizes a round sphere.
\end{lemma}
\begin{proof}[Proof (cf. {\cite[pp.~349 -- 350]{kuwertschaetzle2004}} and {\cite[proof of Thm.~1.2]{rupp2023}}).]
  For the sake of contradiction, suppose that $\hat{\Sigma}$ is not compact. Arguing as in \cite[Lemma~6.1]{rupp2024}, one finds $\A(\hat{f}_j)\to\infty$ for $j\to\infty$. Fix any $x_0\notin \hat{f}(\hat{\Sigma})$, denote by $I$ the inversion in a sphere with radius $1$ centered in $x_0$ and denote $\overline{\Sigma}=I(\hat{f}(\hat{\Sigma}))\cup\{0\}$. Using \eqref{eq:kusch-rem-pt-sing-1} as well as $\W(\hat{f}_j)\leq 8\pi-\beta$ and, by \eqref{eq:int-asq}, $\sup_{j\in\N}\int_{\S^2}|\hat{A}_j|^2\dd\hat{\mu}_j<\infty$, one finds the following. By \cite[Lemma 5.1]{kuwertschaetzle2004}, $\overline{\Sigma}$ is a compact smooth Willmore sphere with $\W(\overline{\Sigma}) \leq 8\pi-\beta$, and thus a round sphere by Bryant's classification in \cite{bryant1984}. Then $\hat{f}(\hat{\Sigma}) = I(\overline{\Sigma}\setminus\{0\})$ is necessarily a plane since it is unbounded, contradicting \eqref{eq:no-plane-condition}.

  Therefore, $\hat{\Sigma}$ is compact and thus, for $j$ sufficiently large, $\hat{\Sigma}(j)=\hat{\Sigma}$. As $\phi_j\colon\hat{\Sigma}(j)\to\phi_j(\hat{\Sigma})\subseteq\S^2$ is a diffeomorphism, arguing as in \cite[Lemma~4.3]{kuwertschaetzle2001}, we may w.l.o.g. write $\hat{\Sigma}=\S^2$ and the convergence in the statement is proved.
  
  Finally, the smooth convergence yields $\W(\hat{f})<8\pi$. Again, by \cite{bryant1984}, $\hat{f}$ necessarily parametrizes a round sphere.
\end{proof}

\begin{proposition}\label{prop:spherical-shrinker}
  Let $c_0<0$ and $\lambda\geq 0$, consider a maximal $(c_0,\lambda)$-Helfrich flow $f\colon[0,T)\times\S^2\to\R^3$ with $\int_{\S^2}H\dd\mu\geq 0$ on $[0,T)$ and $\W(f(t))\leq 8\pi-\beta$ for all $0\leq t<T$ and some $\beta>0$. Then $T<\infty$, the construction in \Cref{lem:ex-blow-up} applies and $r_j\to 0$.

  Moreover, there exist diffeomorphisms $\phi_j\colon\S^2\to\S^2$ such that, after passing to a subsequence,
  \begin{equation}\label{eq:conv-rd-sp}
    \frac{f(t_j+r_j^4\tilde{c})-x_j}{r_j} \circ\phi_j \to \hat{f} \quad\text{in the $C^{\infty}(\S^2,\R^3)$-topology}
  \end{equation}
  where $\hat{f}\colon\S^2\to\R^3$ parametrizes a round sphere. Especially, $\H_{c_0,\lambda}(f(t))\searrow 4\pi$ and $\A(f(t))\to 0$ for $t\nearrow T$.
\end{proposition}
\begin{proof}
  By \Cref{lem:max-ex-time}, we find $T<\infty$. Using $\int_{\S^2}H\dd\mu\geq 0$ and $c_0<0$, an elementary computation yields
  \begin{equation}
    0\leq \W(f(t)) = \H_{c_0}(f(t)) + \frac12c_0\int_{\S^2}H\dd\mu\Big|_t -\frac14c_0^2\A(f(t)) \leq \H_{c_0,\lambda}(f_0) - \frac14c_0^2\A(f(t)).
  \end{equation}
  So $\A(f(t))\leq 4\H_{c_0,\lambda}(f_0)/c_0^2$ for all $0\leq t<T$ and \Cref{lem:ex-blow-up} applies. Furthermore, $t_j\nearrow T<\infty$ and (i) in \Cref{lem:ex-blow-up} yield $r_j\to 0$.

  Arguing exactly as in the proof of \Cref{prop:conc-limit}, there exists a non-empty, complete surface $\hat{\Sigma}$ without boundary and a proper immersion $\hat{f}\colon\hat{\Sigma}\to\R^3$ such that, for $\hat{\Sigma}(j)=\{p\in\hat{\Sigma}:|\hat{f}(p)|<j\}$ there exist $u_j\in C^{\infty}(\hat{\Sigma}(j),\R^3)$ and diffeomorphisms $\phi_j\colon\hat{\Sigma}(j)\to\phi_j(\hat{\Sigma}(j))\subseteq\S^2$ such that 
  \begin{equation}\label{eq:conv-rd-sp-1}
    \frac{f(t_j+r_j^4\tilde{c})-x_j}{r_j} \circ\phi_j=\hat{f}+u_j\quad\text{in $\hat{\Sigma}(j)$}
  \end{equation}
  and $\|\hat{\nabla}^mu_j\|_{L^{\infty}(\hat{\Sigma}(j))}\to 0$ as $j\to\infty$, for all $m\in\N_0$. Moreover, using $r_j\to 0$, as in the proof of (iii) in \Cref{prop:conc-limit}, one finds that $\hat{f}$ is a Willmore immersion. Using (iii) in \Cref{lem:ex-blow-up}, \eqref{eq:conv-rd-sp-1} yields
  \begin{equation}\label{eq:conv-rd-sp-3}
    \int_{B_1(0)} |\hat{A}|^2\dd\hat{\mu} > 0.
  \end{equation}
  By \Cref{lem:kusch-rem-pt-sing}, using $\W((f(t_j+r_j^4\tilde{c})-x_j)/r_j)=\W(f(t_j+r_j^4\tilde{c}))\leq 8\pi-\beta$, we may w.l.o.g. write $\hat{\Sigma}=\S^2$ and the convergence in \eqref{eq:conv-rd-sp} follows. Moreover, $\hat{f}$ necessarily parametrizes a round sphere. Furthermore, since $r_j\to 0$,
  \begin{equation}\label{eq:conv-rd-sp-4}
    \A(f(t_j+r_j^4\tilde{c})) = r_j^2 \A\Big(\frac{f(t_j+r_j^4\tilde{c})-x_j}{r_j}\circ \phi_j\Big) \to 0 \cdot \A(\hat{f}) = 0.
  \end{equation}
  By \eqref{eq:conv-rd-sp}, as $\hat{f}$ is a round sphere, $\W(f(t_j+r_j^4\tilde{c})) \to 4\pi$. Combined with \eqref{eq:conv-rd-sp-4}, one also finds $\H_{c_0,\lambda}(f(t_j+r_j^4\tilde{c}))\to 4\pi$ and the full convergence as $t\nearrow T$ then follows from \Cref{rem:en-dec}.
\end{proof}

\subsection{Excluding curvature concentration for positive $c_0$}\label{subsec:rneq0}

In this subsection we show that, at least for the Helfrich flow of topological spheres, the case $r=0$ in \Cref{prop:conc-limit} cannot occur. This is crucial in our \L ojasiewicz-Simon convergence argument in the next sections --- indeed, to put it simply, one could say that if $r=0$, all information on $c_0$ and $\lambda$ is lost in the concentration limit.

First of all, with the following two results, roughly speaking, we observe that there is no ``actual'' curvature concentration resulting in an unbounded concentration limit. That is, for the case $r=0$, we show that the radii $r_t$ in \Cref{lem:ex-blow-up} asymptotically behave like $\sqrt{\A(f(t))}$ as $t\nearrow T$.

To this end, we first establish that any radius of balls where curvature concentration occurs is suitably lower-bounded by the area if the surface is energetically close to a round sphere. Then, we show that the case $r=0$ already yields that the flow is energetically close to round spheres as $t\nearrow T$.

\begin{lemma}\label{lem:lb-rt}
  Let $\alpha\in(0,8\pi)$. Then there exist $\varepsilon,\eta>0$ with the following property. If $f\colon\S^2\to\R^3$ is an immersion with
  \begin{equation}
    \int_{\S^2}|A^0|^2\dd\mu < \varepsilon
  \end{equation}
  and $B_r(x)\subseteq\R^3$ with $\int_{B_r(x)}|A|^2\dd\mu\geq\alpha$, one finds
  \begin{equation}
    r^2\geq \eta \A(f).
  \end{equation}
\end{lemma}
\begin{proof}
  For the sake of contradiction, suppose that there exist sequences $f_j\colon\S^2\to\R^3$ of immersions and of balls $B_{r_j}(x_j)\subseteq\R^3$ satisfying
  \begin{equation}\label{eq:lb-rt-1}
    \int_{\S^2}|A^0_j|^2\dd\mu_j\to 0,\quad \int_{B_{r_j}(x_j)} |A_j|^2\dd\mu_j\geq\alpha
  \end{equation}
  and $r_j^2/\A(f_j)\to 0$. By invariances with respect to scalings and translations, we may w.l.o.g. suppose that $0\in f_j(\S^2)$ and $\A(f_j)=1$ for all $j\in\N$. Then $r_j\to 0$.

  Notably, by \cite[Lemma~1]{topping1998}, $\mathrm{diam}(f_j)$ is uniformly bounded. That is, there exists $R>0$ with $f_j(\S^2)\subseteq B_R(0)$ for all $j\in\N$. By \eqref{eq:lb-rt-1}, $B_{r_j}(x_j)\cap B_R(0)\supseteq B_{r_j}(x_j)\cap f_j(\S^2)\neq\emptyset$. Especially, since $r_j\to 0$, after passing to a subsequence, we may w.l.o.g. suppose that $x_j\to x$ for some $x\in\R^3$.

  Denoting by $V_j$ the integer rectifiable varifold induced by $f_j$ (for instance cf. \cite[Section~1.1 and Example~2.4]{scharrer2022} for the relevant definitions), after passing to a subsequence, by Allard's compactness theorem in \cite[Theorem~6.4]{allard1972}, one finds $V_j\to V$ weakly as varifolds where $V$ is an integer rectifiable varifold with locally bounded first variation and generalized mean curvature vector $\vec{H}_V$. Further, using $\mathrm{supp}(\mu_{V_j})\subseteq B_R(0)$ for all $j\in\N$ and the weak* convergence $\mu_{V_j}\rightharpoonup^*\mu_V$ as Radon measures on $\R^3$, one finds $\A(V)=1$ and thus $V\neq 0$. Therefore, by \cite[Equation~(A.18)]{kuwertschaetzle2004}, $\W(V)\geq 4\pi$. Moreover, for $\delta>0$, using \cite[Theorem~6.3]{mantegazza1996} for the lower semi-continuity under varifold convergence, and $\W(f_j)\to 4\pi$ by \eqref{eq:lb-rt-1} and \eqref{eq:int-a0sq},
  \begin{align}
    \frac14\int_{\R^3\setminus B_{\delta}(x)} |\vec{H}_V|^2\dd\mu_V &\leq \frac14\liminf_{j\to\infty} \int_{\R^3\setminus B_{\delta}(x)} H_j^2\dd\mu_j \leq 4\pi - \limsup_{j\to\infty} \int_{B_{r_j}(x_j)} H_j^2\dd\mu_j\\
    &=4\pi - \limsup_{j\to\infty} \int_{B_{r_j}(x_j)} \big(\frac12|A_j|^2-\frac12|A^0_j|^2\big)\dd\mu \leq 4\pi-\frac12\alpha,
  \end{align}
  using \eqref{eq:a-a0-h}. So, taking the supremum with respect to $\delta>0$, one finds
  \begin{equation}
    4\pi \leq \W(V)=\frac14\int_{\R^3} |\vec{H}_V|^2\dd\mu_V \leq 4\pi-\frac12\alpha,
  \end{equation}
  a contradiction. 
\end{proof}

\begin{corollary}\label{cor:cor1-rj0}
  Let $c_0,\lambda>0$ and consider a $(c_0,\lambda)$-Helfrich flow $f\colon[0,T)\times\S^2\to\R^3$ satisfying
  \begin{equation}
    \H_{c_0,\lambda}(f_0) < \frac{2\lambda}{c_0^2+2\lambda}8\pi.
  \end{equation}
  Suppose that, for the sequences $(r_j)_{j\in\N}$, $t_j\nearrow T$ and $(x_j)_{j\in\N}$ in \Cref{lem:ex-blow-up,prop:conc-limit}, $\lim_{j\to\infty}r_j=0$. Then $T<\infty$, $\W(f(t))\to 4\pi$ and $\A(f(t))\to 0$ as $t\nearrow T$. Moreover, there exists a family of radii $(r_t)_{t\in[0,T)}$ and a constant $C>1$ with
  \begin{equation}
    \frac{1}{C}\A(f(t))\leq r_t^2\leq C \A(f(t))
  \end{equation} 
  and 
  \begin{equation}
    \alpha < \scurv(t,r_t) < \hat{c}\overline{\varepsilon}
  \end{equation}
  where $\alpha>0$ is as in \eqref{eq:ex-bu-1}, and $\hat{c}$ and $\overline{\varepsilon}$ are as in \Cref{prop:life}.
\end{corollary}
\begin{proof}
  We first argue that the concentration limit $\hat{f}$ in \Cref{prop:conc-limit} necessarily parametrizes a round sphere. Note that $r_j\to 0$ especially yields that $\hat{f}$ is a Willmore immersion. Using \eqref{eq:will-below-8pi}, for some $\beta>0$,
  \begin{equation}\label{eq:cor1-rjo0--1}
    \W((f(t_j+r_j^4\tilde{c})-x_j)/r_j)=\W(f(t_j+r_j^4\tilde{c})) \leq 8\pi-\beta.
  \end{equation}
  Thus, by \Cref{lem:kusch-rem-pt-sing}, we may write $\hat{\Sigma}=\S^2$ and $\hat{f}_j=f_j(\tilde{c})$ converges smoothly on $\S^2$ to a round sphere, after reparametrization and passing to a subsequence. Particularly, 
  \begin{equation}
    \W_0(f(t_j+r_j^4\tilde{c}))=\int_{\S^2}|A^0|^2\dd\mu \Big|_{t=t_j+r_j^4\tilde{c}} \to 0\quad\text{and}\quad \A(f(t_j+r_j^4\tilde{c})) = r_j^2\A(\hat{f}_j) \to 0
  \end{equation}
  for $j\to\infty$. With \Cref{lem:sm-en}, fixing any $\delta>0$, after passing to a subsequence, this improves to
  \begin{equation}\label{eq:cor1-rj0-1}
    \sup_{t\in[t_j+r_j^4\tilde{c},t_j+r_j^4\tilde{c}+\delta]\cap[0,T)}\int_{\S^2}|A^0|^2\dd\mu\to 0\quad\text{for $j\to \infty$}.
  \end{equation}
  \begin{claim}\label{cl:t-sm-infty}
    $T<\infty$.
  \end{claim}
  To this end, for the sake of contradiction, suppose that $T=\infty$. By \eqref{eq:sm-en-linfty-estimate}, 
  \begin{align}
    \int_{t_j+r_j^4\tilde{c}}^{t_j+r_j^4\tilde{c}+\delta} \|A^0\|_{L^{\infty}}^4\dd\tau \to 0\label{eq:cor1-rj0--2}
  \end{align}  
  as $j\to\infty$. Since by \eqref{eq:will-below-8pi},
  \begin{align}
    \Big|\frac{\dd}{\dd t} \A(f(t))\Big| &= -\int_{\S^2}H\langle\partial_tf,\nu\rangle\dd\mu \leq \big(4\W(f(t))\big)^{\frac12}\Big(\int_{\S^2}|\partial_tf|^2\dd\mu\Big)^{\frac12} \\
    &\leq C \Big(-\frac{\dd}{\dd t}\H_{c_0,\lambda}(f(t))\Big)^{\frac12},
  \end{align}
  one finds
  \begin{align}
    &\sup_{t\in[t_j+r_j^4\tilde{c},t_j+r_j^4\tilde{c}+\delta]} \A(f(t)) \\
    &\quad\leq \A(f(t_j+r_j^4\tilde{c})) + C\sqrt{\delta}\sqrt{\H_{c_0,\lambda}(f(t_j+r_j^4\tilde{c}))-\lim_{\tau\nearrow T}\H_{c_0,\lambda}(\tau)} \to 0\label{eq:cor2-rj0-2}
  \end{align}
  for $j\to\infty$. Thus, using \Cref{prop:ev-eq}, \eqref{eq:flow-eq}, \eqref{eq:Gauss-curv}, and $\int_{\S^2}K\dd\mu=4\pi$ by Gauss-Bonnet,
  \begin{align}
    &\V(f(t_j+r_j^4\tilde{c}+\delta))-\V(f(t_j+r_j^4\tilde{c})) \\
    &= \int_{t_j+r_j^4\tilde{c}}^{t_j+r_j^4\tilde{c}+\delta}\Big(\int_{\S^2} \big(|A^0|^2H -c_0|A^0|^2+\frac12c_0H^2-(\frac12c_0^2+\lambda)H\big)\dd\mu\Big)\dd\tau \\
    &= \int_{t_j+r_j^4\tilde{c}}^{t_j+r_j^4\tilde{c}+\delta}\Big(\int_{\S^2}|A^0|^2H\dd\mu + 2c_0\int_{\S^2}K\dd\mu - (\frac12c_0^2+\lambda)\int_{\S^2}H\dd\mu \Big)\dd\tau\\
    &\geq c_0 8\pi \delta - \int_{t_j+r_j^4\tilde{c}}^{t_j+r_j^4\tilde{c}+\delta} \sqrt{\A(f(\tau))}\sqrt{4\W(f(\tau))} \Big(\|A^0\|_{L^{\infty}}^2+\frac12c_0^2+\lambda\Big)\dd\tau \\
    &\geq c_08\pi\delta - C\sqrt{\delta}\sup_{\tau\in[t_j+r_j^4\tilde{c},t_j+r_j^4\tilde{c}+\delta]}\sqrt{\A(f(\tau))} \Big(\int_{t_j+r_j^4\tilde{c}}^{t_j+r_j^4\tilde{c}+\delta}\|A^0\|_{L^{\infty}}^4\dd\tau+\delta(\frac12c_0^2+\lambda)^2\Big)^{\frac12}\\
    &\geq c_08\pi\delta - C(\lambda^2,c_0^4,\delta)\sup_{\tau\in[t_j+r_j^4\tilde{c},t_j+r_j^4\tilde{c}+\delta]}\sqrt{\A(f(\tau))},
  \end{align}
  also using \eqref{eq:cor1-rj0--2}. This gives a uniform, positive lower bound for $\V(f(t_j+r_j^4\tilde{c}+\delta))$, thus by the isoperimetric inequality contradicting $\A(f(t_j+r_j^4\tilde{c}+\delta))\to 0$. This proves \Cref{cl:t-sm-infty}.

  Using \Cref{cl:t-sm-infty}, \eqref{eq:cor1-rj0-1} and \eqref{eq:cor2-rj0-2} yield
  \begin{equation}\label{eq:cor1-rj0-3}
    \int_{\S^2}|A^0|^2\dd\mu\Big|_{t}\to 0\quad\text{and}\quad\A(f(t))\to 0\quad\text{for $t\nearrow T$}.
  \end{equation}
  With the construction in the proof of \Cref{lem:ex-blow-up}, using the above time-uniform bounds on $\W(f(t))$ and $\A(f(t))$, there exist radii $(r_t)_{t\in[0,T)}$ and $\alpha>0$ such that $\alpha<\scurv(t,r_t)<\hat{c}\overline{\varepsilon}$. Moreover, by \eqref{eq:cl-bu-1} and \eqref{eq:will-below-8pi},
  \begin{equation}
    r_t\leq \frac{1}{\pi}\sqrt{\W(f(t))} \sqrt{\A(f(t))}\leq \sqrt{8/\pi} \sqrt{\A(f(t))}\quad\text{for $0\leq t<T$}.
  \end{equation}
  Furthermore, \eqref{eq:cor1-rj0-3} and \Cref{lem:lb-rt} yield the reversed estimate.
\end{proof}

\begin{proposition}[Sub-convergence to round spheres]\label{prop:prop1-rj0}
  Let $c_0,\lambda>0$ and consider a $(c_0,\lambda)$-Helfrich flow $f\colon[0,T)\times\S^2\to\R^3$ satisfying
  \begin{equation}
    \H_{c_0,\lambda}(f_0) < \frac{2\lambda}{c_0^2+2\lambda}8\pi.
  \end{equation}
  Suppose that, for $(r_j)_{j\in\N}$, $t_j\nearrow T$ and $(x_j)_{j\in\N}$ as in \Cref{lem:ex-blow-up}, $\liminf_{j\to\infty}r_j=0$.

  Then $T<\infty$ and, for every sequence $\tau_j\nearrow T$, there are sequences $\tilde{r}_j>0$ and $\tilde{x}_j\in\R^3$ such that $(f(\tau_j)-\tilde{x}_j)/\tilde{r}_j$ smoothly converges to a round sphere up to reparametrization and passing to a subsequence, for $j\to\infty$. Particularly, there exists $t_0\in(0,T)$ with
  \begin{equation}\label{eq:a0sqH-has-a-sign}
    \int_{\S^2}|A^0|^2H\dd\mu\Big|_{t}\geq 0\quad\text{for all $t_0\leq t<T$}.
  \end{equation}
\end{proposition}
\emph{Remark.} The remarkable new insight in this proposition is the ``full'' sub-convergence to suitable round spheres one obtains for $t\nearrow T$. So far, using only the blow-up construction in \Cref{lem:ex-blow-up}, in general, we are only able to prove the existence of \emph{some} sequence of times for which the concentration limit is a round sphere. Here we show way more: For any given sequence of times converging to $T$, there exists a subsequence such that, along this subsequence, the Helfrich flow behaves like a suitably rescaled and translated round sphere. And this is exactly what we need in order to give integrals as in \eqref{eq:a0sqH-has-a-sign} a sign for $t\approx T$.
\begin{proof}[Proof of \Cref{prop:prop1-rj0}.]
  Consider some $0<\delta<1$ to be chosen later and consider $(r_t)_{t\in[0,T)}$ as in \Cref{cor:cor1-rj0}. Since $T<\infty$ and as $t+r_t^4\tilde{c}<T$ by \Cref{prop:life}, one finds $r_t\to 0$. Particularly, passing to a subsequence, we may w.l.o.g. suppose that $s_j\vcentcolon=\tau_j-r_{\tau_j}^4\delta >0$ for all $j\in\N$.   
  \begin{claim}\label{cl:est-rt}
    There is $\hat{C}>1$ independent of $\delta$ with $\displaystyle\frac{1}{\hat{C}} < \liminf_{j\to\infty}\frac{r_{\tau_j}}{r_{s_j}}\leq \limsup_{j\to\infty}\frac{r_{\tau_j}}{r_{s_j}} < \hat{C}$.
  \end{claim}
  By \Cref{cor:cor1-rj0}, one finds for some constant $\tilde{C}>1$ 
  \begin{equation}
    \frac{1}{\tilde{C}^2}\frac{\A(f(\tau_j))}{\A(f(s_j))}\leq \frac{r_{\tau_j}^2}{r_{s_j}^2} \leq \tilde{C}^2\frac{\A(f(\tau_j))}{\A(f(s_j))}.
  \end{equation}
  Moreover, using \eqref{eq:will-below-8pi} and arguing as in \eqref{eq:cor2-rj0-2}, one finds
  \begin{align}
    |\A(f(s_j))-\A(f(\tau_j))|&\leq C\sqrt{r_{\tau_j}^4\delta} \sqrt{\H_{c_0,\lambda}(f(s_j))-\H_{c_0,\lambda}(f(\tau_j))} \leq r_{\tau_j^2}o(1)\\
    &\leq \A(f(\tau_j))o(1)
  \end{align}
  as $j\to\infty$. Altogether, one concludes \Cref{cl:est-rt}.
  \begin{claim}\label{cl:bd-xij}
    There exists $\delta\in(0,1)$ such that $\tau_j=s_j+r_{s_j}^4\zeta_j$ with $0<\zeta\leq \zeta_j\leq \tilde{c}$.
  \end{claim}
  First of all,
  \begin{equation}
    \tau_j=s_j+r_{s_j}^4\zeta_j = \tau_j-r_{\tau_j}^4\delta+r_{s_j}^4\zeta_j.
  \end{equation}
  Thus, using \Cref{cl:est-rt}, after passing to a subsequence, 
  \begin{equation}
    \frac{\delta}{\hat{C}^4}\leq \zeta_j\leq \hat{C}^4\delta
  \end{equation}
  for all $j\in\N$. Thus, choosing $\delta\leq \frac{\tilde{c}}{\hat{C}^4}$, \Cref{cl:bd-xij} follows.

  Write $\tilde{r}_j=r_{s_j}$. Proceeding as in \Cref{cl:est-rt}, one finds that, for some $M>0$,
  \begin{equation}
    \limsup_{j\to\infty} \frac{r_{s_j+\tilde{r}_j^4\tilde{c}}}{r_{s_j}} < M.
  \end{equation}
  Thus, proceeding as in the proof of \Cref{lem:ex-blow-up}, there is $N=N(M)$ such that one can choose $\tilde{x}_j\in\R^3$ with $\int_{B_{\tilde{r}_j}(\tilde{x}_j)}|A|^2\dd\mu\Big|_{t=s_j+\tilde{r}_j^4\tilde{c}}\geq\frac{\alpha}{2N}$. 

  Now define $f_j\colon[\zeta,\tilde{c}]\times\S^2\to\R^3$, $f_j(t,p)=\frac{1}{\tilde{r}_j}(f(s_j+\tilde{r}_j^4t,p)-\tilde{x}_j)$. Proceeding as in the proof of \Cref{prop:conc-limit}, especially using the lower bound $\zeta_j\geq \zeta>0$ in \Cref{cl:bd-xij} as well as $\tilde{r}_j\to 0$, one finds that, up to passing to a subsequence and reparametrization, $\hat{f}_j\vcentcolon=\frac{1}{\tilde{r}_j}(f(\tau_j)-\tilde{x}_j)=f_j(\zeta_j)$ converges to some proper Willmore immersion $\hat{f}\colon\hat{\Sigma}\to\R^3$ smoothly on compact subsets of $\R^3$. Moreover, by \Cref{cl:est-rt} and \Cref{cor:cor1-rj0}, one finds that $\A(\hat{f}_j)$ is uniformly bounded from above and below. Thus, arguing as in the proof of \Cref{cor:cor1-rj0}, one can w.l.o.g. write $\hat{\Sigma}=\S^2$ and obtain that $\hat{f}$ necessarily parametrizes a round sphere. 
  
  Then \eqref{eq:a0sqH-has-a-sign} can be proved by contradiction using the above sub-convergence result and a standard subsequence argument.
\end{proof}

\begin{corollary}\label{cor:cor2-rj0}
  Let $c_0>0$ and $\lambda > 0$. Let $f\colon[0,T)\times\Sigma\to\R^3$ be a maximal $(c_0,\lambda)$-Helfrich flow on a compact, closed and oriented surface $\Sigma$ satisfying
  \begin{equation}\label{eq:as-stab-en-thresh}
    \H_{c_0,\lambda}(f_0) < \frac{2\lambda}{c_0^2+2\lambda}8\pi
  \end{equation}
  and $\hat{f}\colon\hat{\Sigma}\to\R^3$ a concentration limit as in \Cref{prop:conc-limit}. Then $r_j\to r\in (0,\infty)$.
\end{corollary}
\begin{proof}
  For the sake of contradiction, suppose that $r_j\to r=0$. Using \cref{prop:prop1-rj0}, especially \eqref{eq:a0sqH-has-a-sign}, one finds for $t\in [t_0,T)$, with a similar computation as in \Cref{cor:cor1-rj0} and using \eqref{eq:will-below-8pi},
  \begin{align}
    \frac{\dd}{\dd t}\V(f(t)) &= 8\pi c_0 + \int_{\S^2}|A^0|^2H\dd\mu - (\frac12c_0^2+\lambda)\int_{\S^2}H\dd\mu\\
    &\geq 8\pi c_0 - (\frac12c_0^2+\lambda)\sqrt{32\pi} \sqrt{\A(f(t))}
  \end{align}
  so that, using $\A(f(t))\to 0$ for $t\nearrow T$ by \Cref{cor:cor1-rj0}, $\liminf_{t\nearrow T}\V(f(t))>0$. This however contradicts the isoperimetric inequality! 
\end{proof}

\section{Asymptotic stability of critical points}\label{sec:loja}

Again, in this section, $\Sigma$ is some compact, orientable and closed surface of genus $g\in\N_0$. One can prove the following \L ojasiewicz-Simon gradient inequality for the locally area--constrained Helfrich functional
\begin{equation}
  \H_{c_0,\lambda}(f)=\H_{c_0}(f) + \frac12\lambda\A(f)
\end{equation}
where $f\colon\Sigma\to\R^3$ is an immersion. Write
\begin{equation}
  \nabla\H_{c_0,\lambda}(f) = \nabla\H_{c_0}(f)+\frac12\lambda\nabla\A(f).
\end{equation}

\begin{theorem}\label{thm:loja}
  Let $c_0,\lambda\in\R$. If $f\colon\Sigma\to\R^3$ is a $(c_0,\lambda)$-Helfrich immersion, then there exist $C,r>0$ and $\theta\in(0,\frac12]$ such that, for all immersions $h\in W^{4,2}(\Sigma,\R^3)$ with $\|h-f\|_{W^{4,2}}\leq r$, one has
  \begin{equation}
    |\H_{c_0,\lambda}(f)-\H_{c_0,\lambda}(h)|^{1-\theta} \leq C \|\nabla\H_{c_0,\lambda}(h)\|_{L^2(\dd\mu_h)}.
  \end{equation}
\end{theorem}
As $\nabla\H_{c_0,\lambda}(f)$ only differs from $\nabla\W_0(f)$ by some lower-order terms, \Cref{thm:loja} is proved analogously to \cite[Theorem 3.1]{chillfasangovaschaetzle2009}. Note that, in \cite{lengeler2018}, using these arguments, a version of \Cref{thm:loja} is already established for embedded surfaces in the case where $\lambda=0$. Therefore, a detailed proof of \Cref{thm:loja} can be safely omitted. 

As in \cite[Lemma~4.1]{chillfasangovaschaetzle2009}, \cite[Lemma~7.9]{rupp2023} and \cite[Lemma~5.8]{rupp2024}, one uses \Cref{thm:loja} to deduce the following

\begin{lemma}\label{lem:as-stab}
  Let $f_W\colon\Sigma\to\R^3$ be a $(c_0,\lambda)$-Helfrich immersion and $k\in\N$ with $k\geq 4$, $\delta>0$. Then there exists $\varepsilon>0$ depending on $f_W$ such that, if $f\colon[0,T)\times\Sigma\to\R^3$ is a maximal $(c_0,\lambda)$-Helfrich flow starting in $f_0$ satisfying
  \begin{enumerate}[(i)]
    \item $\|f_0-f_W\|_{C^{k,\alpha}}<\varepsilon$ for some $\alpha\in(0,1)$ and
    \item $\H_{c_0,\lambda}(f(t))\geq\H_{c_0,\lambda}(f_W)$ whenever $\|f(t)\circ\Phi(t)-f_W\|_{C^k}\leq\delta$, for certain diffeomorphisms $\Phi(t)\colon\Sigma\to\Sigma$,
  \end{enumerate}
  then $T=\infty$ and there are diffeomorphisms $\widetilde{\Phi(t)}\colon\Sigma\to\Sigma$ such that $f(t)\circ\widetilde{\Phi(t)}$ converges smoothly to a $(c_0,\lambda)$-Helfrich immersion $f_{\infty}$ with $\H_{c_0,\lambda}(f_{\infty})=\H_{c_0,\lambda}(f_W)$.
\end{lemma}

Finally, with \Cref{lem:as-stab} at our disposal, we can prove the following result which yields convergence if the concentration limit of a Helfrich flow is compact and if $r>0$ in \Cref{prop:conc-limit}. Note that here we do not yet make the assumption $g=0$, i.e. restricting the analysis to spherical immersions.

\begin{theorem}\label{thm:as-stab}
  Let $c_0>0$ and $\lambda > 0$. Let $f\colon[0,T)\times\Sigma\to\R^3$ be a maximal $(c_0,\lambda)$-Helfrich flow on a compact, closed and oriented surface $\Sigma$ of genus $g\in\N_0$ and $\hat{f}\colon\hat{\Sigma}\to\R^3$ a concentration limit as in \Cref{prop:conc-limit} with $r>0$. If $\hat{\Sigma}$ has a compact component, then $T=\infty$ and, as $t\to\infty$, the flow converges smoothly after reparametrization to a $(c_0,\lambda)$-Helfrich immersion $f_{\infty}$ with $\H_{c_0,\lambda}(f_{\infty})=\H_{c_0,\lambda}(\hat{f})$.
\end{theorem}
\begin{proof}[Proof of \Cref{thm:as-stab}.]
  Let $(x_j)_{j\in\N}$, $t_j\nearrow T$ and $r_j\to r\in(0,\infty)$ be as in \Cref{prop:conc-limit}. Since $\hat{\Sigma}$ has a compact component, one argues as in \cite[Lemma~4.3]{kuwertschaetzle2001} that $\hat{\Sigma}$ is connected and diffeomorphic to $\Sigma$. So, using (i) in \Cref{prop:conc-limit}, we may w.l.o.g. take $\hat{\Sigma}=\Sigma$ and there are diffeomorphisms $\Phi_j\colon\Sigma\to\Sigma$ such that $\hat{f}_j\circ\Phi_j\to\hat{f}$ smoothly on $\Sigma$. By (iii) in \Cref{prop:conc-limit}, $\hat{f}$ is an $(rc_0,r^2\lambda)$-Helfrich immersion. Therefore, we can fix $\varepsilon=\varepsilon(\hat{f})$ as in \Cref{lem:as-stab}.

  Since $r>0$, $r_j\to r$ and $\hat{f}_j\circ\Phi_j\to\hat{f}$, there exists $j_0\in\N$ with
  \begin{equation}\label{eq:pf-as-stab-2}
    \frac{|r_{j_0}-r|}{r} \|\hat{f}_{j_0}\circ\Phi_{j_0}\|_{C^{4,\alpha}}<\frac{\varepsilon}{2},\quad \|\hat{f}_{j_0}\circ\Phi_{j_0}-\hat{f}\|_{C^{4,\alpha}}<\frac{\varepsilon}{2}.
  \end{equation}  
  Using \Cref{lem:par-scal}, consider now the re-scaled $(rc_0,r^2\lambda)$-Helfrich flow given by
  \begin{equation}
    \widetilde{f}_{j_0} = \frac{1}{r}\Bigl(f(t_{j_0}+r^4\ \cdot,\cdot)-x_{j_0}\Bigr)\circ\Phi_{j_0}\colon\Big[0,\frac{T-t_{j_0}}{r^4}\Big)\times\Sigma\to\R^3.
  \end{equation}
  Particularly, for any $0\leq t<(T-t_{j_0})/r^4$, using \Cref{rem:en-dec}, one finds
  \begin{align}
    \H_{rc_0,r^2\lambda}(\widetilde{f}_{j_0}(t)) &\geq \lim_{s\nearrow (T-t_{j_0})/r^4}\H_{rc_0,r^2\lambda}(\widetilde{f}_{j_0}(s)) = \lim_{s\nearrow T} \H_{c_0,\lambda}(f(s)) \\
    &= \lim_{j\to\infty}\H_{c_0,\lambda}(f(t_j+r_j^4\tilde{c}))=\lim_{j\to\infty} \H_{r_jc_0,r_j^2\lambda}(\hat{f}_j\circ\Phi_j) = \H_{rc_0,r^2\lambda}(\hat{f}),
  \end{align}
  repeatedly using \Cref{lem:par-scal} as well as $f_j(\tilde{c})=\hat{f}_j$ and the smooth convergence $\hat{f}_j\circ\Phi_j\to\hat{f}$. Since $\hat{f}_{j_0}=\frac{1}{r_{j_0}}(f(t_{j_0}+r_{j_0}^4\tilde{c})-x_{j_0})$, \eqref{eq:pf-as-stab-2} yields
  \begin{equation}
    \|\widetilde{f}_{j_0}(r_{j_0}^4/r^4\tilde{c})-\hat{f}\|_{C^{4,\alpha}} <\varepsilon.
  \end{equation}
  Thus, by \Cref{lem:as-stab}, the flow $\widetilde{f}_{j_0}$ exists globally and converges after reparametrization by appropriate diffeomorphisms to an $(rc_0,r^2\lambda)$-Helfrich immersion $f_{\infty}$ with $\H_{rc_0,r^2\lambda}(f_{\infty})=\H_{rc_0,r^2\lambda}(\hat{f})$. The theorem is proved after again applying \Cref{lem:par-scal}.
\end{proof}

\section{Proof of the main results}\label{sec:proof-main}

Following the strategy of \cite[proof of Theorem 5.2]{kuwertschaetzle2004}, using the above findings, we can carry out the proof of the main results.

\begin{proof}[Proof of \Cref{thm:main}.]
  First, by \eqref{eq:en-thresh}, the case $\lambda=0$ is trivial. Indeed, if $\lambda=0$, then $\H_{c_0}(f_0)=0$, so particularly $f_0$ parametrizes a round sphere of radius $2/c_0$ due to a result by Hopf, cf. \cite[Theorem 2.1 in Chapter VI]{hopf1983}. Particularly, the $(c_0,0)$-Helfrich flow is stationary and thus exists globally. Therefore, we may suppose that $\lambda>0$.
  
  Second, again using \eqref{eq:en-thresh}, one can w.l.o.g. suppose that, for some $\varepsilon>0$,
  \begin{equation}\label{eq:pf-main-1}
    \H_{c_0}(f_0)+\frac12\lambda\A(f_0)=K\leq\frac{2\lambda}{2\lambda+c_0^2}8\pi-\varepsilon.
  \end{equation}
  Indeed, to justify this assumption, one distinguishes the two cases where $f_0$ is or is not a $(c_0,\lambda)$-Helfrich immersion. In the former case, the flow is stationary and the statement of the theorem is trivial. In the latter case, \Cref{rem:en-dec} yields that \eqref{eq:pf-main-1} is satisfied if we replace $f_0$ by $f(\delta)$ for any $\delta\in(0,T)$ --- which we may w.l.o.g. do since the statement of the theorem concerns only the asymptotic behavior of $f$ for $t\nearrow T$.

  Consider a concentration limit $\hat{f}\colon\hat{\Sigma}\to\R^3$ as constructed in \Cref{prop:conc-limit}. If $\hat{\Sigma}$ is compact, the claim follows from \Cref{thm:as-stab}, using \Cref{cor:cor2-rj0}.

  Suppose that $\hat{\Sigma}$ is not compact. Arguing as in \cite[Lemma~6.1]{rupp2024}, one finds $\A(\hat{f}_j)\to\infty$ for $j\to\infty$. \Cref{prop:conc-limit} then yields that $\hat{f}$ is a Willmore immersion. Moreover, by \eqref{eq:will-below-8pi} and \eqref{eq:pf-main-1}, one finds for some $\beta>0$
  \begin{equation}
    \W(f(t)) \leq 8\pi-\beta\quad\text{for all $0\leq t<T$}.
  \end{equation}
  Using the first statement in (ii) in \Cref{prop:conc-limit}, \Cref{lem:kusch-rem-pt-sing} yields that $\hat{f}$ parametrizes a round sphere --- a contradiction! 
\end{proof}

\begin{proof}[Proof of \Cref{prop:neg-result}.]
  The upper bound on the maximal existence time $T$ is obtained in \Cref{cor:fin-time-sing}. By the choice of the constants $C$ and $\alpha_1$ in \Cref{cor:fin-time-sing}, \Cref{lem:sm-en,prop:pos-av-mc-for-sm-en} yield that $\int_{\S^2}H\dd\mu>0$ on $[0,T)$ and, also using \eqref{eq:int-a0sq}, 
  \begin{equation}
    \W(f(t))\leq \frac{1}{2}\big(8\pi+\W_0(f(t))\big)\leq \frac{1}{2}(8\pi+\alpha_1)<8\pi
  \end{equation}
  for all $t\in[0,T)$. Thus, \Cref{prop:spherical-shrinker} implies the second part of the claim.
\end{proof}
\appendix

\section{Higher-order interpolation estimates}\label{app:ho-int}

Following \cite{kuwertschaetzle2001,kuwertschaetzle2002}, denote by $\phi*\psi$ any multilinear form which depends on $\phi$ and $\psi$ in a blilinear way, where $\phi$ and $\psi$ are tensors on $\Sigma$. Particularly, $|\phi*\psi|\leq C|\phi||\psi|$ for a universal constant $C>0$ and $\nabla(\phi*\psi)=\nabla\phi*\psi+\phi*\nabla\psi$. Further, for $m\in\N_0$ and $r\geq 2$, denote by $P_r^m(A)$ any term of type
\begin{equation}
  P_r^m(A)=\sum_{i_1+\dots+i_r=m} \nabla^{i_1}A*\dots*\nabla^{i_r}A.
\end{equation} 
Further, $P_1^m(A)$ denotes a contraction of the tensor $\nabla^mA$ with respect to the metric $g$. One essentially proceeds as in \cite[Propositions~3.3~and~4.5]{kuwertschaetzle2002} to prove the following proposition. In order to keep track of the dependence of the constants on the parameters $c_0$ and $\lambda$ (cf. \Cref{rem:dep-c0-lam}), we include the details for the reader's convenience. Note that the estimates for the dependence on $\lambda$ are very similar to the version of \cite[Propositions~3.3~and~4.5]{kuwertschaetzle2002} proved in \cite[Appendix~B]{rupp2024}. 

\begin{proposition}\label{prop:ho-est}
  Let $c_0,\lambda\in\R$ and $f\colon[0,T)\times\Sigma\to\R^3$ be a $(c_0,\lambda)$-Helfrich flow. Further, consider $\gamma$ as in \eqref{eq:gamma}. For $m\in\N_0$, $s\geq 2m+4$ and $\phi=\nabla^mA$, one finds
  \begin{align}
    \frac{\dd}{\dd t}\int_{\Sigma}|\phi|^2\gamma^s\dd\mu + \ & \frac12\int_{\Sigma} |\nabla^2\phi|^2\gamma^s\dd\mu \leq C\bigl(c_0^4+\lambda^2+\|A\|^4_{L^{\infty}(\{\gamma>0\})}\bigr) \int_{\Sigma}|\phi|^2\gamma^s\dd\mu  \\
    &+ C\bigl(1+c_0^4+\lambda^2+\|A\|^4_{L^{\infty}(\{\gamma>0\})}\bigr)\int_{\{\gamma>0\}}|A|^2\dd\mu\label{eq:ho-est}
  \end{align}
  for some $C=C(s,\Lambda)$ with $\Lambda$ as in \eqref{eq:gamma}.
\end{proposition}
\begin{proof}
  In the following estimates, constants $C$ may change from line to line and depend on $s$ and $\Lambda$. Proceeding as in \cite[Proposition~2.4]{kuwertschaetzle2002}, also cf. \cite[Lemma~B.1]{rupp2024}, using $\partial_tf=\xi\nu$ with
  \begin{equation}\label{eq:ho-est-0}
    \xi=-\Delta H + P_3^0(A) + (\lambda+c_0^2)P_1^0(A) + c_0P_2^0(A),
  \end{equation}
  one finds for $m\in\N_0$
  \begin{align}
    \partial_t(\nabla^mA) &+\Delta^2(\nabla^m A)= P_3^{m+2}(A) + P_5^m(A) \\
    &+ (\lambda+c_0^2)(P_1^{m+2}(A)+P_3^m(A)) + c_0(P_2^{m+2}(A)+P_4^m(A)).\label{eq:ho-est-1}
  \end{align}
  With $\phi=\nabla^mA$ and $Y=\partial_t\phi+\Delta^2\phi$, \cite[Lemma~3.2]{kuwertschaetzle2002} yields for some $C=C(s)$
  \begin{align}
    \frac{\dd}{\dd t}&\int_{\Sigma}|\phi|^2\gamma^s\dd\mu+\int_{\Sigma}|\nabla^2\phi|^2\gamma^s\dd\mu\\
    &\leq 2\int_{\Sigma}\langle Y,\phi\rangle\gamma^s\dd\mu+\int_{\Sigma} A*\phi*\phi*\xi\ \gamma^s\dd\mu + \int_{\Sigma}|\phi|^2s\gamma^{s-1}\partial_t\gamma\dd\mu\\
    &\quad + C \int_{\Sigma}|\phi|^2\gamma^{s-4}\big(|\nabla\gamma|^4+\gamma^2|\nabla^2\gamma|^2\big)\dd\mu + C\int_{\Sigma}|\phi|^2\big(|\nabla A|^2+|A|^4\big)\gamma^s\dd\mu. \label{eq:ho-est-3}
  \end{align}
  We now analyze the terms on the right-hand side of \eqref{eq:ho-est-3} and show that they can be suitably estimated and absorbed on the left-hand side to yield the claim, i.e. \eqref{eq:ho-est}. For the first, second and last term on the right-hand side of \eqref{eq:ho-est-3}, using \eqref{eq:ho-est-0} and \eqref{eq:ho-est-1},
  \begin{align}
    2&\int_{\Sigma}\langle Y,\phi\rangle\gamma^s\dd\mu+\int_{\Sigma} A*\phi*\phi*\xi\ \gamma^s\dd\mu + C\int_{\Sigma}|\phi|^2\big(|\nabla A|^2+|A|^4\big)\gamma^s\dd\mu\\
    &=\int_{\Sigma} \big(P_3^{m+2}(A)+P_5^m(A)\big)*\phi\ \gamma^s\dd\mu  + (\lambda+c_0^2) \int_{\Sigma}\big(P_1^{m+2}(A)+P_3^m(A)\big)*\phi\ \gamma^s\dd\mu\\
    &\quad + c_0 \int_{\Sigma}\big(P_2^{m+2}(A)+P_4^m(A)\big)*\phi\ \gamma^s\dd\mu.
  \end{align}
  The term $\int_{\Sigma} \big(P_3^{m+2}(A)+P_5^m(A)\big)*\phi\ \gamma^s\dd\mu$ can be absorbed in \eqref{eq:ho-est} using \cite[Equation~(4.15)]{kuwertschaetzle2002}. Moreover, by Young's inequality, for $\eta>0$,
  \begin{align}
    (\lambda+c_0^2) \int_{\Sigma}P_1^{m+2}(A)*\phi\ \gamma^s\dd\mu \leq \eta \int_{\Sigma}|\nabla^2\phi|^2\gamma^s\dd\mu + C(\eta)(\lambda^2+c_0^4) \int_{\Sigma} |\phi|^2\gamma^s\dd\mu.
  \end{align}
  Furthermore, interpolating with $k=m$, $r=4$ in \cite[Corollary~5.5]{kuwertschaetzle2002}, 
  \begin{align}
    &(\lambda+c_0^2) \int_{\Sigma}P_3^{m}(A)*\phi\ \gamma^s\dd\mu \leq C (|\lambda|+c_0^2)\|A\|^2_{L^{\infty}(\{\gamma>0\})} \Big(\int_{\Sigma}|\phi|^2\gamma^s\dd\mu+\int_{\{\gamma>0\}}|A|^2\dd\mu\Big)\\
    &\quad \leq C\big(\lambda^2+c_0^4+\|A\|^4_{L^{\infty}(\{\gamma>0\})}\big) \Big(\int_{\Sigma}|\phi|^2\gamma^s\dd\mu+\int_{\{\gamma>0\}}|A|^2\dd\mu\Big).
  \end{align}
  Additionally, interpolating with $k=m$ and $r=5$ in \cite[Corollary~5.5]{kuwertschaetzle2002} and using Young's inequality with $p=4$ and $p'=\frac{4}{3}$,
  \begin{align}
    c_0\int_{\Sigma}P_4^m(A)*\phi\ \gamma^s\dd\mu &\leq C |c_0|\|A\|_{L^{\infty}(\{\gamma>0\})}^3\Big(\int_{\Sigma}|\phi|^2\gamma^s\dd\mu+\int_{\{\gamma>0\}}|A|^2\dd\mu\Big)\\
    &\leq C(c_0^4+\|A\|_{L^{\infty}(\{\gamma>0\})}^4)\Big(\int_{\Sigma}|\phi|^2\gamma^s\dd\mu+\int_{\{\gamma>0\}}|A|^2\dd\mu\Big).
  \end{align}
  Finally, for $\eta>0$, applying \cite[Corollary~5.5]{kuwertschaetzle2002} with $k=m+1$ and $r=3$, using $s\geq 2k=2m+4\geq 2m+2$ and repeatedly using Young's inequality,
  \begin{align}
    &c_0\int_{\Sigma}P_2^{m+2}(A)*\phi\ \gamma^s\dd\mu \leq C |c_0|\Big|\int_{\Sigma} \nabla^{m+2}A*\nabla^mA*A\ \gamma^s\dd\mu\Big|\\
    &\quad + C|c_0| \|A\|_{L^{\infty}(\{\gamma>0\})}\Big(\int_{\Sigma}|\nabla\phi|^2\gamma^s\dd\mu+\int_{\{\gamma>0\}}|A|^2\dd\mu\Big)\\
    &\leq \eta \int_{\Sigma}|\nabla^2\phi|^2\gamma^s\dd\mu + C(\eta) c_0^2\|A\|_{L^{\infty}(\{\gamma>0\})}^2\int_{\Sigma}|\phi|^2\gamma^s\dd\mu\\
    &\quad + C(1+c_0^4+\|A\|_{L^{\infty}(\{\gamma>0\})}^4) \int_{\{\gamma>0\}}|A|^2\dd\mu + C|c_0| \|A\|_{L^{\infty}(\{\gamma>0\})}\int_{\Sigma}|\nabla\phi|^2\gamma^s\dd\mu.
  \end{align}
  For the last term in the above, using \cite[Lemma~5.1]{kuwertschaetzle2002} with $p=q=2r=2$ and $\alpha=0$, $\beta=1$ and $t=0$,
  \begin{align}
    |c_0| &\ \|A\|_{L^{\infty}(\{\gamma>0\})}\int_{\Sigma}|\nabla\phi|^2\gamma^s\dd\mu \\
    &\leq C |c_0| \|A\|_{L^{\infty}(\{\gamma>0\})} \Big(\int_{\Sigma}|\phi|^2\gamma^s\dd\mu\Big)^{\frac12}\Big(\int_{\Sigma}|\nabla^2\phi|^2\gamma^s\dd\mu\Big)^{\frac12}\\
    &\quad + C |c_0| \|A\|_{L^{\infty}(\{\gamma>0\})} \Big(\int_{\Sigma}|\phi|^2\gamma^s\dd\mu\Big)^{\frac12}\Big(\int_{\Sigma}|\nabla\phi|^2\gamma^{s-2}\dd\mu\Big)^{\frac12}\\
    &\leq \eta \int_{\Sigma}|\nabla^2\phi|^2\gamma^s\dd\mu + C(\eta) c_0^2\|A\|^2_{L^{\infty}(\{\gamma>0\})} \int_{\Sigma}|\phi|^2\gamma^s\dd\mu + C(\eta) \int_{\Sigma}|\nabla\phi|^2\gamma^{s-2}\dd\mu.
  \end{align}
  Here, the first three terms can now easily be absorbed/estimated to fit \eqref{eq:ho-est} while for the last term, one uses \cite[Equation~(3.11)]{kuwertschaetzle2002} which reads
  \begin{equation}\label{eq:kusch02-3.11}
    \int_{\Sigma}|\phi|^2\gamma^{s-4}\dd\mu + \int_{\Sigma}|\nabla\phi|^2\gamma^{s-2}\dd\mu \leq \eta\int_{\Sigma}|\nabla^2\phi|^2\gamma^s\dd\mu + C_{\eta} \int_{\{\gamma>0\}}|A|^2\gamma^{s-4-2m}\dd\mu.
  \end{equation}
  Now turning to the third term on the right-hand side in \eqref{eq:ho-est-3}, using $\gamma=\widetilde{\gamma}\circ f$, one finds
  \begin{equation}
    \int_{\Sigma}|\phi|^2\gamma^{s-1}\partial_t\gamma\dd\mu = \int_{\Sigma}|\phi|^2\gamma^{s-1}\langle D\widetilde{\gamma}\circ f,\nu\rangle \big(-\Delta H + P_3^0(A)+(\lambda+c_0^2)P_1^0(A)+c_0P_2^0(A)\big)\dd\mu.
  \end{equation}
  By the computation on \cite[p. 320]{kuwertschaetzle2002}, using \eqref{eq:kusch02-3.11}, for $\eta>0$,
  \begin{align}
    \int_{\Sigma}|\phi|^2\gamma^{s-1}\langle D\widetilde{\gamma}\circ f,\nu\rangle\Delta H\dd\mu &\leq C \int_{\Sigma} |\nabla\phi|^2\gamma^{s-2}\dd\mu + C \int_{\Sigma} |\phi|^2\gamma^{s-4}\dd\mu \\
    &\quad + \int_{\Sigma} \big(P_3^{m+2}(A)+P_5^m(A)\big)*\phi \ \gamma^s\dd\mu \\
    &\leq \eta \int_{\Sigma} |\nabla^2\phi|^2\gamma^s\dd\mu + C(\eta) \int_{\{\gamma>0\}} |A|^2\gamma^{s-4-2m}\dd\mu  \\
    &\quad + \int_{\Sigma} \big(P_3^{m+2}(A)+P_5^m(A)\big)*\phi \ \gamma^s\dd\mu.
  \end{align}
  Using Young's inequality,
  \begin{align}
    &\int_{\Sigma}|\phi|^2\gamma^{s-1}\langle D\widetilde{\gamma}\circ f,\nu\rangle P_3^0(A)\dd\mu \leq C\int_{\Sigma} |\phi|^2|A|^4\gamma^s\dd\mu + C \int_{\Sigma} |\phi|^2\gamma^{s-4}\dd\mu;\\
    &\int_{\Sigma} |\phi|^2\gamma^{s-1}\langle D\widetilde{\gamma}\circ f,\nu\rangle (\lambda+c_0^2)P_1^0(A) \leq C (\lambda^2+c_0^4)\int_{\Sigma}|\phi|^2\gamma^{s}\dd\mu + C \int_{\Sigma}|\phi|^2|A|^2\gamma^{s-2}\dd\mu\\
    &\quad \leq C(\lambda^2+c_0^4+\|A\|^4_{L^{\infty}(\{\gamma>0\})})\int_{\Sigma}|\phi|^2\gamma^s\dd\mu + C\int_{\Sigma}|\phi|^2\gamma^{s-4}\dd\mu;\\
    &\int_{\Sigma} |\phi|^2\gamma^{s-1}\langle D\widetilde{\gamma}\circ f,\nu\rangle c_0P_2^0(A)\dd\mu \leq Cc_0^2\int_{\Sigma}|\phi|^2\gamma^s|A|^2\dd\mu + C\int_{\Sigma}|\phi|^2\gamma^{s-2}|A|^2\dd\mu\\
    &\quad\leq C (c_0^4+\|A\|^4_{L^{\infty}(\{\gamma>0\})})\int_{\Sigma}|\phi|^2\gamma^s\dd\mu + C \int_{\Sigma} |\phi|^2\gamma^{s-4}\dd\mu.
  \end{align}
  By \eqref{eq:kusch02-3.11}, all the above terms can be absorbed in \eqref{eq:ho-est}, using $s\geq 2m+4$ and $\gamma\leq 1$. For the penultimate term on the right-hand side in \eqref{eq:ho-est-3}, one proceeds exactly as on \cite[p.~321]{kuwertschaetzle2002}. This finishes the proof of \eqref{eq:ho-est}.
\end{proof}

\paragraph*{Acknowledgments.}
The author would like to thank Anna Dall’Acqua for helpful discussions and comments.

\paragraph*{Declaration of interests.}
The author declares that they have no known competing financial interests or personal relationships that could have appeared to influence the work reported in this paper.


\end{document}